\documentclass[12pt]{amsart} 

\usepackage[margin=3cm]{geometry}
\usepackage[english]{babel}
\usepackage{hyperref}
\usepackage{amsmath}
\usepackage{amsthm}
\usepackage{amssymb}
\usepackage{ucs}
\usepackage{enumerate}
\usepackage[utf8]{inputenc}
\usepackage[T1]{fontenc}
\usepackage{color}
\usepackage{tikz}
\usepackage{diagbox}
\usepackage{tikz-cd}
\usepackage{comment}

\usepackage[draft,inline,nomargin,index]{fixme} 
\fxsetup{theme=color,mode=multiuser}
\FXRegisterAuthor{pf}{apf}{\color{blue}PF}
\FXRegisterAuthor{ip}{aip}{\color{orange}IP}
\FXRegisterAuthor{km}{akm}{\color{red}KM}
\FXRegisterAuthor{flm}{aflm}{\color{purple}FLM}

\DeclareMathOperator{\onelip}{1-\mathrm{Lip}}

  \newcommand{\E}{\mathbb E}
  \newcommand{\R}{\mathbb R}
  \newcommand{\N}{\mathbb N}
  
  \newcommand{\Z}{\mathbb Z}
  \newcommand{\C}{\mathbb C}
  \newcommand{\LL}{\mathrm L}

  \newcommand{\ot}{\otimes}

  \newcommand{\B}{\mathcal B}
  \renewcommand{\H}{\mathcal H}
  \newcommand{\BH}{\B(\H)}

\newcommand{\Sub}{\mathrm{Sub}}

  \renewcommand{\leq}{\leqslant}
  \renewcommand{\geq}{\geqslant}
  \newcommand{\abs}[1]{\left\lvert #1\right\rvert}
  \newcommand{\norm}[1]{\left\lVert #1\right\rVert}

  \newcommand{\impl}{\Rightarrow}

  \newcommand{\la}{\left\langle}
  \newcommand{\ra}{\right\rangle}

\newtheorem{thm}{Theorem}[section]

\newtheorem{lem}[thm]{Lemma}
\newtheorem{prop}[thm]{Proposition}

\theoremstyle{definition}

\newtheorem{remark}[thm]{Remark}

\newtheorem{df}[thm]{Definition}
\newtheorem*{ack}{Acknowledgments}
\newtheorem{example}[thm]{Example}

\newcommand{\bimod}[3]{{\vphantom{#2}}_{#1}{#2\;\!}_{#3}}
\newcommand{\ucp}[4]{\mathrm{UCP}_{#1,#2}(#3,#4)}

\title[Spaces of UCP maps and subalgebras of von Neumann algebras]{Spaces of UCP maps and subalgebras of von Neumann algebras}

\author{Pierre Fima}
\address{Pierre Fima
\newline
Universit\'e Paris Cit\'e, Sorbonne Universit\'e, CNRS, IMJ-PRG, F-75013 Paris, France.}
\email{pierre.fima@imj-prg.fr}

\author{François Le Maître}

\address{François Le Maître
\newline
Institut de Math\'ematiques de Bourgogne, UMR 5584 CNRS, Universit\'e de Bourgogne, 21000 Dijon, France}
\email{flemaitre@math.cnrs.fr }

\author{Kunal Mukherjee}

\address{Kunal Mukherjee
\newline
Indian Institute of Technology Madras, Chennai 600 036, India}
\email{kunal@iitm.ac.in}

\author{Issan Patri}

\address{Issan Patri
\newline
Theoretical Statistics and Mathematics Unit, Indian Statistical Institute, Delhi Centre, 7 S. J. S. Sansanwal Marg, New
Delhi 110016, India}
\email{issanp@isid.ac.in}

\begin{document}

\begin{abstract}
In this paper, we establish that several natural topologies on the space of state-preserving unital completely positive maps coincide and that make it a Polish space. We then focus on the subspace of state-preserving conditional expectations and analyse its topology in detail, recovering the Haagerup-Winslow result that it aligns with the Effros--Mar\'echal topology on the space of von Neumann subalgebras. 

This correspondence is then applied to structural classes of subalgebras, including amenable, Haagerup and weakly amenable subalgebras. Among other consequences, we demonstrate the closedness of amenable subalgebras admitting state-preserving conditional expectations and analyze the semicontinuity and failure of continuity of the Cowling--Haagerup constant as a function on subalgebras. Finally, we investigate the space of von Neumann subalgebras that are not the image of state preserving conditional expectations for a fixed faithful normal state. For several important classes of von Neumann algebras---such as type ${\rm III}_\lambda$ factors with $0 \le \lambda < 1$ and type ${\rm III}_1$ factors with a state whose centraliser is infinite dimensional---we show that the subalgebras lacking state-preserving conditional expectations form an open and dense subset. Thus, in these settings, the generic subalgebra is not the range of state preserving conditional expectation.
\end{abstract}

\date{}


\maketitle

\section{Introduction}

This paper is a continuation of our study of the Maréchal (or Effros-Maréchal) topology on the space of 
von Neumann subalgebras of a given von Neumann algebra \cite{FLMMP24}.
Its starting point is a characterization of convergence in this space in terms 
of pointwise convergence of conditional expectations due to Haagerup and Winsløw
(building on work of Tsukada \cite{tsukadaStrongLimitNeumann1985})
which we first state in a special case. 

\begin{thm}[{\cite[Corollary~2.12]{haagerupEffrosMarechalTopologySpace1998}}]\label{thm: HW pointwise cv}
Given a finite von Neumann algebra $(M,\tau)$ with separable predual, 
a sequence of subalgebras $N_n$ converges to $N$ in the Maréchal topology 
if and only if for all $x\in M$, one has
$\mathbb E_{N_n}(x)\to \mathbb E_{N}(x)$ strongly, where $\mathbb E_N$ denotes the 
($\tau$-preserving) conditional expectation onto $N$. 
\end{thm}

It is a result of Maréchal that the Maréchal topology is Polish, although a
full proof only appeared recently, in our first paper on this topic  \cite{FLMMP24}.
Combining this with Theorem~\ref{thm: HW pointwise cv}, it is immediate that the space of 
conditional expectations is a Polish space when endowed with the topology of pointwise strong convergence. 
Our first aim in this paper is to place this result in the more global framework of ($\tau$-preserving) ucp maps. 
We work at the level of possibly non finite von Neumann algebras, which is the general setup 
for Haagerup and Winsløw's Theorem~\ref{thm: HW pointwise cv}, as follows- Let $M$ be a von Neumann algebra with separable predual, fix a faithful normal state $\varphi$, 
and denote by
$\mathcal E_\varphi$ of all von Neumann subalgebras $N$ of $M$ which admit 
a $\varphi$-preserving conditional expectation. 
Haagerup and Winsløw's Theorem \ref{thm: HW pointwise cv} becomes a characterization of Maréchal convergence of elements of $\mathcal E_{\varphi}$ as pointwise convergence in the strong topology of their $\varphi$-preserving conditional expectations. 

If we further fix another von Neumann algebra $N$ with separable predual and a faithful normal state $\psi$, we arrive 
at the most general notion of maps considered in this paper, which are state-preserving ucp maps $(M,\varphi)\to (N,\psi)$.
Such maps are automatically normal (see Remark \ref{RmkNormal}), and our first theorem reads as follows.

\begin{thm}[see Proposition~\ref{ConservativeucpisPolish}]
    The space of state-preserving ucp maps $(M,\varphi)\to(N,\psi)$ is a Polish space 
    for the topology of pointwise strong convergence.
\end{thm}

We then observe that when $(N,\psi)=(M,\varphi)$,
by identifying elements of $\mathcal E_{\varphi}$ to their conditional expectation,
the space $\mathcal E_{\varphi}$ becomes closed in the space of ucp maps $(M,\varphi)\to(M,\varphi)$ (see Theorem~\ref{Subalgpolish}), 
while the general form of Theorem~\ref{thm: HW pointwise cv} only yields that it is $G_\delta$ therein.
We additionally provide a more direct proof of Theorem~\ref{thm: HW pointwise cv} (see Theorem~\ref{thm: marechal same as one lip}).\\

Next, we make use of the nicer topology provided by pointwise strong convergence so as to study 
various approximation properties in the space $\mathcal E_{\varphi}$, where their behavior is much tamer
than in the whole space of subalgebras. 

\begin{thm}
Given a von Neumann algebra $M$ endowed with a faithful normal state $\varphi$, the following spaces
defined closed subspaces of $\mathcal E_\varphi$:
\begin{enumerate}[(i)]
    \item \label{item: amenable closed} the space of all $M\in\mathcal E_{\varphi}$ which are amenable; 
    \item \label{item: haagerup closed}the space of all $M\in\mathcal E_\varphi$ satisfying the Haagerup property;
    \item \label{item: Lambdacb lsc}the space of all $M\in\mathcal E_\varphi$ whose Cowling-Haagerup constant is at most $C\geq 0$.
\end{enumerate}
\end{thm}
A few remarks are in order: Item \eqref{item: amenable closed} was obtained in \cite{TYH25} in the finite case, with a more complicated proof
based on hypertraces. Here we directly use semidiscreteness in a more general setup (see Proposition~\ref{prop: amenable is closed v1}). Further, we are able to generalize this further to relative amenability, but we have to follow yet another approach which also makes use of the topology on the space of UCP maps (see Theorem~\ref{thm:rel-amenable-closed-core-proof}).
Concerning Item~\eqref{item: haagerup closed}, we also obtain a more general relative version, see Theorem~\ref{thm: rel haagerup is closed}. Finally, Item~\eqref{item: Lambdacb lsc} can be restated as the fact that the Cowling-Haagerup constant is lower semi-continuous, and we further show that it is not continuous, even at the level of the space of subgroups, and use this to describe the 
space of weakly amenable subalgebras, which is $F_\sigma$ but neither open nor closed (see Section~\ref{sec: weakly amenable}).
Applications to Kechris' spaces of subequivalence relations are also presented.
\\

Finally, we study von Neumann subalgebras that lack state preserving conditional expectations, showing that for a von Neumann algebra with a state $\varphi$ having infinite dimensional centraliser, the space of subalgebras without a $\varphi$-preserving conditional expectation is open and dense (see Theorem~\ref{Nocondexpecdense}). 
For instance, the theorem applies to type ${\rm III}_\lambda$ factors with $\lambda\in [0,1)$ (see Remark~\ref{remarkoncentralizerhypothesis}).
We also characterize non finite von Neumann algebras as those for which the set of subalgebras which do not carry 
a conditional expectation is dense (see Theorem~\ref{thm: no condex is dense in finite}).\\

\paragraph{\textbf{Overview of the paper.}} 
In Section~\ref{sec:prelim}, we collect the necessary preliminaries on Polish spaces, von Neumann algebras, and the Effros--Mar\'{e}chal topology on the space of von Neumann subalgebras. In
Section~\ref{sec:ucp}, we study the space of state-preserving unital completely
positive maps between von Neumann algebras equipped with faithful normal states,
establishing that several natural topologies on this space coincide and that this space equipped with any of these topologies is a Polish space. In Section~\ref{sec:cond_exp}, we specialise to the space $\mathcal{E}_\varphi$
of von Neumann subalgebras admitting a $\varphi$-preserving conditional expectation,
showing that it is a closed Polish subspace and that its topology coincides with the
restriction of the Effros--Mar\'{e}chal topology. In Section~\ref{sec:subalgebras}, we
apply this topological framework to study the properties of structural classes of subalgebras. Finally,  in
Section~\ref{sec:no_cond_exp}, we analyse the complement of $\mathcal{E}_\varphi$ in
the space of subalgebras of $M$.

\begin{ack}
The present paper is an offspring of a long term project on the 
Maréchal topology. 
We have had the chance to carry out this 
project in several places dedicated to international collaboration, and 
we thus gratefully acknowledge funding from ISI Delhi, SERB (SRG/2022/001717),  IRL ReLaX, IEA GAOA, IIT Madras (IoE Phase II, SB22231267MAETWO008573), ICISE, CEFIPRA (6101-1) and ANR-25-CE40-5010 (Project CroCQG).
\end{ack}

\section{Preliminaries}\label{sec:prelim}

\subsection{Polish spaces} A topological space which is separable and whose topology admits a compatible metric is called a \textbf{Polish space}.
Since metrizable spaces are separable iff they are second-countable (in other words, their topology admits a countable basis), any closed subset of a Polish space is Polish for the induced topology.
Also note that a compact Hausdorff space is Polish iff it is second-countable, because Urysohn's metrization theorem provides  a compatible metric which has to be complete by compactness.

We note the following general example of Polish space.

\begin{thm}	\label{thm: one lip is polish}
Let $(X,d_X)$ and $(Y,d_y)$ be separable complete metric spaces. Consider the space of $1$-Lipschitz maps from $X$ to $Y$ 
$$
\onelip(X,Y)=
\{f: X\to Y:\forall x_1,x_2\in X, d_Y(f(x_1),f(x_2))\leq d_X(x_1,x_2)\}.
$$
Then $\onelip(X,Y)$ is a Polish space for the topology of pointwise convergence.

 Moreover, if $D$ is a dense subset of $X$, then for a sequence $(f_n)$ of $1$-Lipschitz functions we have $f_n\to f$ if and only if $f_n(x)\to f(x)$ for all $x\in D$. 
\end{thm}
\begin{proof}
Let $D\subseteq X$ be a countable dense subset of $X$. Consider the space $\onelip(D,Y)$, then for all $f\in Y^D$, we have $f\in \onelip(D,Y)$ if and only if for every $x_1,x_2\in D$, 
$$d_Y(f(x_1),f(x_2))\leq d_X(x_1,x_2).$$
We thus have that $\onelip(D,Y)$ is a closed subset of the Polish space $Y^D$. Now recall that since $d_Y$ is complete and $D$ is dense in $X$, every $1$-Lipschitz map from $D$ to $X$ uniquely extends to a $1$-Lipschitz map from $X$ to $Y$. 

Denote by $\Phi:\onelip(D,Y)\to\onelip(X,Y)$ the injective map which associates to $f\in\onelip(D,Y)$ its extension in $\onelip(X,Y)$. Then $\Phi$ is a bijection whose inverse is continuous by definition of the product topology. 
We will show that $\Phi$ is continuous as well. Indeed, suppose $(f_n)$ is a sequence of elements of $\onelip(D,Y)$ such that $f_n\to f$, and let $x\in X$. For all $d\in D$,  we have 
\begin{align*}
d_Y(\Phi(f_n)(x),\Phi(f)(x))&\leq d_Y(\Phi(f_n)(x),f_n(d))\\&+d_Y(f_n(d),f(d))+d_Y(f(d),\Phi(f)(x))\\
&\leq 2d_X(x,d)+d_Y(f_n(d),f(d)).
\end{align*}
So for a given $\epsilon>0$, if we pick $d\in D$ such that $d_X(x,d)<\frac \epsilon 4$, then since $f_n(d)\to f(d)$ we will have $d_Y(\Phi(f_n)(x),\Phi(f)(x))<\epsilon$ for large enough $n$. We conclude that $\Phi(f_n)\to \Phi(f)$, so $\Phi$ is indeed continuous. 

Thus, $\onelip(X,Y)$ is homeomorphic to the Polish space $\onelip(D,Y)$, so it is a Polish space itself. Further, the moreover part of the theorem is simply a reformulation of the fact that the map $\Phi$ is a homeomorphism. 
\end{proof}

\subsection{von Neumann algebras}

All von Neumann algebras considered in the present paper are supposed to have a separable predual. Given a Banach space $X$, we will denote by $X_1$ its unit ball. We will need the following two well-known Lemmas. We include proofs for the sake of completeness.

\begin{lem}\label{lem: topologies on ball are the same}
	Let $(M,\omega)$ be von Neumann algebra equipped with a faithful normal state, then on the unit ball of $M$, the strong and the $\norm\cdot_2$ topology coincide. Furthermore, the weak and the $\LL^2$-weak topology coincide. 
\end{lem}

\begin{proof}
By GNS construction we may and will assume that $M\subset\BH$, $\xi\in\mathcal{H}$ is a norm $1$ vector such that $\overline{M\xi}=\mathcal{H}$ and $\omega(x)=\langle x\xi,\xi\rangle$ for all $x\in M$. Note that the orthogonal projection $p$ onto $\overline{M'\xi}\subseteq\mathcal{H}$ is in $M''=M$ and, since $p\xi=\xi$, $\omega(1-p)=0$. By faithfulness of $\omega$ we get $p=1$ so $\overline{M'\xi}=\mathcal{H}$. Since $\Vert x\Vert_2=\Vert x\xi\Vert$, it is clear that any subbasic $\Vert\cdot\Vert_2$-open set is strongly open. Let $O:=\{x\in M_1\,:\,\Vert (x-x_0)\eta\Vert<\epsilon\}$ by any subbasic strong open neighborhood of $x_0\in M_1$. Let $a\in M'$ be such that $\Vert \eta-a\xi\Vert<\epsilon/4$ and note that $\Vert (x-x_0)\eta\Vert\leq\Vert a\Vert\,\Vert x-x_0\Vert_2+\epsilon/2$. Hence, 
$$\{x\in M_1\,:\,\Vert x-x_0\Vert_2<\frac{\epsilon}{2\Vert a\Vert}\}\subset O.$$
This shows that the strong and the $\Vert\cdot\Vert_2$ topology coincide on $M_1$. The last statement is even easier to prove and is left to the reader.
\end{proof}

\begin{lem}\label{lem: equiv cv for proj}
	Let $\mathcal H$ be a Hilbert space, $(p_n)_n$ a net of orthogonal projections and $p$ be another orthogonal projection. The following are equivalent: 
	\begin{enumerate}[(i)]
		\item $p_n\to p$ strongly;
		\item For all $\xi\in\mathcal H$, $\norm{p_n\xi}\to\norm{p\xi}$;
		\item $p_n\to p$ weakly.
	\end{enumerate}
\end{lem}
\begin{proof}
	The implications (i)$\impl$(ii) and (i)$\impl$(iii) are clear. We will now show that (iii)$\impl$ (i) and then that (ii)$\impl$(i), which will finish the proof.
	
	(iii)$\impl$(i): suppose $p_n\to p$ weakly. Let $\xi\in\mathcal H$, we have 
	$$\norm{p\xi-p_n\xi}^2=\norm{p\xi^2}+\norm{p_n\xi}^2-2\Re\la p_n\xi,p\xi\ra=\norm{p\xi}^2+\la p_n\xi,\xi\ra-2\Re\la p_n\xi,p\xi\ra$$
	so by weak convergence $\norm{p\xi-p_n\xi}^2\to \norm{p\xi}^2+\la p\xi,p\xi\ra-2\Re\la p\xi,p\xi\ra=0$.
	
	 (ii)$\impl$(i): we suppose that for all $\xi\in\mathcal H$, $\norm{p_n\xi}\to\norm{p\xi}$. For $\xi\in\mathcal H$ we have:
	\begin{equation}
	\label{ineq: show that cv ptwise of norm is same as strong cv for proj}
	\norm{p_n\xi-p\xi}\leq\norm{p_n\xi-p_np\xi}+\norm{p_np\xi)-p\xi}.
	\end{equation}
	Let $\xi'=\xi-p\xi$, then $\norm{p\xi'}=0$ so by our asumption $\norm{p_n\xi'}\to 0$. We thus have $\norm{p_n\xi-p_np\xi}\to 0$, which takes care of the first term in the above inequality. 
	
	For the second term, we have
	\begin{align*}
	\norm{p_np\xi-p\xi}^2&=\norm{p_np\xi}^2+\norm{p\xi}^2-2\Re\la p_np\xi,p\xi\ra
	=\norm{p_np\xi}^2+\norm{p\xi}^2-2\norm{p_np\xi}^2\\
	&=\norm{p\xi}^2-\norm{p_np\xi}^2.
	\end{align*}
	But since $pp\xi=p\xi$, we have $\norm{p_np\xi}\to \norm{p\xi}$ so last member of this equation tends to zero. We conclude that our second term from \eqref{ineq: show that cv ptwise of norm is same as strong cv for proj} also tends to zero as wanted. 
\end{proof}

\subsection{The space of von Neumann subalgebras}

Let $M$ be a von Neumann algebra (with separable predual $M_*$) and $\mathcal{S}(M)$ be the set of von Neumann sublagebras of $M$ with the smallest topology making the maps $\mathcal{S}(M)\rightarrow [0,+\infty)$, $N\mapsto\Vert\omega\vert_N\Vert$ continuous, for all $\omega\in M_*$. This topology has been introduced by Maréchal \cite{marechalTopologieStructureBorelienne1973} who also has shown that $\mathcal{S}(M)$ is Polish. The proof of the Polishness of the topology has been later fixed in \cite{FLMMP24}. We record the following simple Lemmas for later use. The first one is a direct consequence of \cite[Lemma 3.4]{FLMMP24}.

\begin{lem}\label{LemInclusion}
The set $\{(A,B)\in\mathcal{S}(M)\times\mathcal{S}(M)\,:\,A\subseteq B\}$ is closed in $\mathcal{S}(M)\times\mathcal{S}(M)$. In particular, for any $N\in\mathcal{S}(M)$ the subset $\mathcal{S}(N)\subseteq\mathcal{S}(M)$ is closed.
\end{lem}

The following Lemma has been known to specialists for a long time \cite{tsukadaStrongLimitNeumann1985,haagerupEffrosMarechalTopologySpace1998}. We provide below an easy and modern proof for reader's convenience.

\begin{lem}\label{lem: convergence increasing}
Let $M$ be a von Neumann algebra. The following holds.
\begin{enumerate}
\item If $N_n$ is an increasing sequence in $\mathcal{S}(M)$ then $N_n\to_n (\bigcup_k N_k)''$.
\item If $N_n$ is a decreasing sequence  in $\mathcal{S}(M)$ then $N_n\to_n \bigcap_k N_k$.
\end{enumerate}
\end{lem}
\begin{proof}
$(1)$. Let $\omega\in M_*$ and note that, since $N_n\subseteq N_{n+1}\subseteq N:=(\bigcup_k N_k)''$ one has $\Vert\omega\vert_{N_n}\Vert\leq\Vert\omega\vert_{N_{n+1}}\Vert\leq\Vert\omega\vert_{N}\Vert$ for all $n\in N$ hence, the sequence $(\Vert\omega\vert_{N_n}\Vert)_n$ is convergent and $\lim\Vert\omega\vert_{N_n}\Vert\leq\Vert\omega\vert_N\Vert$. Let $x\in N_1$ such that $\Vert \omega\vert_N\Vert=\vert\omega(x)\vert$. By Kaplansky's theorem there is a sequence $(x_k)_k$ in $\bigcup N_n$, $\Vert x_k\Vert\leq 1$, converging strongly to $x$. Since $(N_n)_n$ is increasing, there is an increasing map $\phi\,:\,\N\rightarrow\N$ such that $x_k\in N_{\phi(k)}$ for all $k$. Hence,
$$\Vert\omega\vert_N\Vert=\vert\omega(x)\vert=\lim\vert\omega(x_n)\vert\leq\lim\Vert\omega_{N_{\phi(n)}}\Vert=\lim\Vert\omega_{N_{n}}\Vert\leq \Vert\omega\vert_{N}\Vert.$$
So the sequence $(\Vert\omega\vert_{N_n}\Vert)_n$ converges to  $\Vert\omega\vert_{N}\Vert$. This shows that $N_n\rightarrow N$.

\vspace{0.2cm}

\noindent$(2)$. Let $\omega\in M_*$ and note that, since $N:=\bigcap_k N_k\subseteq N_{n+1}\subseteq N_n$ one has $\Vert\omega\vert_{N}\Vert\leq\Vert\omega\vert_{N_{n+1}}\Vert \leq\Vert\omega\vert_{N_n}\Vert$ for all $n\in N$ hence, the sequence $(\Vert\omega\vert_{N_n}\Vert)_n$ is convergent and $\Vert\omega\vert_N\Vert\leq\lim\Vert\omega\vert_{N_n}\Vert$. For each $n\in\N$ let $x_n\in(N_n)_1$ be such that $\Vert\omega\vert_{N_n}\Vert=\omega(x_n)$. Since $(x_n)_n$ is a sequence in $M_1$ which is $\sigma$-weakly compact, there exists $\phi\,:\,\N\rightarrow\N$ increasing and $x\in M_1$ such that $x_{\phi(n)}\rightarrow x$ $\sigma$-weakly. Fix $n\in\N$ and note that, for all $k\geq n$ one has $x_{\phi(k)}\in N_{\phi(k)}\subset N_{\phi(n)}\subset N_n$. Hence, $x\in N_n$ for all $n\in\N$ therefore $x\in N$ and:
$$\Vert\omega\vert_N\Vert\geq\vert\omega(x)\vert=\lim_n\vert\omega(x_{\phi(n)})\vert=\lim_n\Vert\omega\vert_{N_{\phi(n)}}\Vert=\lim\Vert\omega\vert_{N_n}\Vert\geq\Vert\omega\vert_N\Vert.$$
It concludes the proof.
\end{proof}

\section{Spaces of ucp maps}\label{sec:ucp}

In this section, we fix once and for all two von Neumann algebras $M$ and $N$ equipped with normal faithful states (nfs) $\varphi$ and $\rho$ respectively, and 
denote by 
$\ucp\varphi\rho MN$ the space of all state-preserving (also called conservative) unital completely positive map (ucp) maps $M\rightarrow N$.

\begin{df}	
A state-preserving ucp map  $\Phi\,:\, M\rightarrow N$ is called a $(\varphi,\rho)$\textbf{-Markov map} if $\Phi\circ \sigma_t^\varphi = \sigma_t^\rho\circ\Phi$ for all $t\in\mathbb{R}$, $\sigma_t^\varphi$ and $\sigma_t^\rho$ denote the modular groups of $\varphi$ and $\rho$ respectively.
\end{df}

Note that every ucp map is a contraction with respect to the operator norm and, by the Kadison-Schwarz inequality \cite[Prop. 3.3]{paulsenCompletelyBoundedMaps2003}, $$\Phi(x)^*\Phi(x)\leq \Phi(x^*x)\text{ for any ucp map }\Phi\,:\, M\rightarrow N\text{ and any }x\in M.$$ 
Thus, if $\Phi$ is state preserving, then $$\norm{\Phi(x)}_{2,\rho}^2=\rho(\Phi(x)^*\Phi(x))\leq\rho(\Phi(x^*x))=\varphi(x^*x)=\norm x_{2,\varphi}^2, \text{ }\forall\text{ }x\in M,$$
so that state-preserving ucp maps are contractions with the respective $2$-norms on $M$ and $N$ induced by the states $\varphi$ and $\rho$ respectively.

\begin{remark}\label{RmkNormal}
It is well known that any ucp map $\Phi\,:\,M\rightarrow N$ preserving faithful states $\varphi\in M_*$ and $\rho\in N_*$ is automatically faithful and normal. Indeed, the faithfulness being a direct consequence of the faithfulness of the states, it suffices to show that, for all $\omega\in N_*$, one has $\omega\circ\Phi\in M_*$. Suppose first that there exists $a,b\in N$ such that $b$ is entire with respect to $\sigma_t^\rho$ (see Definition 2.2  in \cite[Chapter~VIII]{Tak2}) and $\omega=b.\rho.a$ i.e. $\omega(x)=\rho(axb)$ for all $x\in N$. Then,
\begin{eqnarray*}
\vert\omega\circ\Phi(x)\vert&=&\rho(a\Phi(x)b)=\vert\rho(\sigma_i^\rho(b)a\Phi(x))\vert\leq\Vert\Phi(x)\Vert_{2,\rho}\Vert a^*\sigma_i^\rho(b)^*\Vert_{2,\rho}\\
&\leq&\Vert x\Vert_{2,\varphi}\Vert a^*\sigma_i^\rho(b)^*\Vert_{2,\rho}\quad\text{for all }x\in M.
\end{eqnarray*}
Hence, $\omega\circ\Phi$ is strongly continuous so $\omega\circ\Phi\in M_*$. Consider now the vector space $X:=\{\omega\in N_*\,:\,\omega\circ\Phi\in M_*\}$ and note that $X$ is norm closed in $M_*$. Moreover, it contains $\{b.\rho.a\,:\,a,b\in N\text{ and }b\text{ is analytic}\}$ which is norm dense in $M_*$ since the set of analytic elements in strongly dense in $N$.
\end{remark}

The preceding discussion allows us to view the set of state-preserving ucp maps from $M$ to $N$ as a subspace of $\onelip((M,\norm\cdot_{2,\varphi}),(N,\norm\cdot_{2,\rho}))$, and since linear maps are determined by their restriction to the unit ball, this can further be identified to a subspace of $\onelip\left(\big(M_1,\norm\cdot_{2,\varphi}\big),\big(N_1,\norm\cdot_{2,\rho}\big)\right)$, which we denote $\onelip_1(M_\varphi,N_\rho)$.

\begin{prop}\label{ConservativeucpisPolish}
The collection of state-preserving ucp maps from $(M,\varphi)$ to $(N,\rho)$ is a closed subspace of $\onelip_1(M_\varphi,N_\rho)$ and therefore forms a Polish space. Furthermore, the set of
$(\varphi,\rho)$-Markov maps is closed therein, hence Polish as well.
\end{prop}
\begin{proof}	
We first note that the space of restrictions of linear maps from $M$ to $N$ to the unit ball of $M$ is precisely the space of maps $\phi: M_1\to N$ such that
$\phi(\lambda x)=\lambda \phi(x)$ for every $x\in M_1$ and $\lambda\in\C$ such that $\lambda x\in M_1$ and $\phi(x+y)=\phi(x)+\phi(y)$ for every $x,y\in M_1$ such that $x+y\in M_1$. 
In particular, we thus have that the space of $1$-Lipschitz linear maps from $(M,\norm\cdot_{2,\varphi})$ to $(N,\norm\cdot_{2,\rho})$ can be identified with a closed subspace of $\onelip_1(M_\varphi,N_\rho)$. 
Recall now that the unit balls of $M$ and $N$ are separable complete metric spaces with respect to the two-norms induced by $\varphi$ and $\rho$ respectively, therefore $\onelip_1(M_\varphi,N_\rho)$ is a Polish space by Theorem \ref{thm: one lip is polish}. 
Further, since states are continuous with respect to associated two-norms, the space of such maps preserving states is also closed. It is also easily seen that maps for which $\phi(1_M)=1_N$ form a closed subspace of $\onelip_1(M_\varphi,N_\rho)$. 
Finally, complete positivity means that for each $n\in\N$ and for every positive $X=(x_{i,j})\in M_n(M)$, and every $(y_1,...,y_n)$ with $y_i\in (N)_1$, we have $\sum_{i,j}\la \phi(X_{ij})y_i,y_j\ra_{\LL^2(N,\rho)}\geq 0$ and by using the continuity of multiplication on $N_1$ with respect to $\norm{\cdot}_{2,\rho}$, it follows that maps which satisfy these conditions form a closed subspace of $\onelip_1(M_\varphi,N_\rho)$.

For the last statement, observe that, since $\sigma_t^\rho$ is strongly continuous for all $t\in\R$, the set of $(\varphi,\rho)$-Markov maps is a closed subset of the space of state-preserving ucp maps.
\end{proof}

\begin{prop}\label{proposition: equivalent topologies on ucp}
Consider the space of state-preserving ucp maps from $(M,\varphi)$ to $(N,\rho)$. Then the following topologies on this space coincide
\begin{enumerate}[(1)]
\item \label{item: pointwise strong star cv} pointwise convergence in strong-$\ast$ topology.
\item \label{item: pointwise strong cv} pointwise convergence in strong topology. 
\item \label{item: one lip topology} topology induced from the space $\onelip_1(M_\varphi,N_\rho)$.
\end{enumerate}
\end{prop}

\begin{proof}
The equivalence of $(1)$ and $(2)$ follows immediately from the fact that ucp maps commute with the adjoint operation while the equivalence of $(2)$ and $(3)$ follows from Lemma \ref{lem: topologies on ball are the same}.\end{proof}

\subsection{Weak containment in the space of ucp maps}

Let ${}_M\mathcal H_N$ and ${}_M\mathcal K_N$ be Hilbert $M$-$N$ bimodules, where $M,N$ are von Neumann algebras, with induced $*$-representations on $M\odot N^{op}$ denoted $\pi_{\mathcal H},\pi_{\mathcal K}: M\odot N^{op}\to \mathbb B(\mathcal H),\mathbb B(\mathcal K)$.

We say that ${}_M\mathcal H_N$ is \textbf{weakly contained} in ${}_M\mathcal K_N$, and write
$$
{}_M\mathcal H_N \prec {}_M\mathcal K_N,
$$
if for every $T\in M\odot N^{op}$, we have
$$
\|\pi_{\mathcal H}(T)\|\leq \|\pi_{\mathcal K}(T)\|.
$$
As is easy to see, equivalently, the $C^*$-seminorm on $M\odot N^{op}$ induced by $\mathcal H$ is dominated by the $C^*$-seminorm induced by $\mathcal K$.

It is well known that state-preserving ucp maps $(M,\varphi)\to (N,\rho)$ are in one-to-one correspondence
with cyclic $M-N$ bimodules $(\mathcal H,\xi)$ such that $\varphi=\la \cdot \xi,\xi\ra$
and $\rho=\la \xi\cdot,\xi\ra$.
Let us denote by $(\mathcal H_\Phi,\xi_\Phi)$ be the cyclic bimodule as above 
obtained from a ucp trace-preserving map $\Phi$, obtained by separation and completion of 
the tensor product $M\odot \LL^2(N,\rho)$ endowed with the pre-Hilbert space structure given by 
$$
\la x_1\odot \xi_1, x_2\odot \xi_2 \ra_\Phi= \la\Phi(x_2^*x_1)\xi_1,\xi_2 \ra_{\LL^2(N,\rho)}.
$$
(with corresponding norm denoted $\|\cdot\|_\Phi$). Then, as above, we have weak containment of cyclic bimodules, $\bimod M {\mathcal H_{\Psi}} N \prec  \bimod M {\mathcal H_{\Phi}} N$, when for every $T\in M\odot N^{op}$, we have $\|\pi_{{\mathcal H}_\Psi}(T)\|\leq \|\pi_{{\mathcal H}_\Phi}(T)\|$.

The following general statement will be useful when considering (co) amenability 
in the space of von Neumann subalgebras admitting a $\varphi$-preserving conditional
expectation. 

\begin{prop}\label{prop: weak-containment-closed}
Given $\Phi\in \ucp\varphi\rho MN$, the set of all $\Psi\in\ucp\varphi\rho MN$
such that $\bimod M {\mathcal H_{\Psi}} N \prec  \bimod M {\mathcal H_{\Phi}} N$
is closed in the pointwise ultraweak  topology,
in particular it is closed in the Polish topology of $\ucp\varphi\rho MN$.
\end{prop}
\begin{proof}

Let $(\Psi_i)_i$ be a net in the set
$$
U:=\{\Lambda\in \ucp\varphi\rho MN\mid {}_M(\mathcal H_\Lambda)_N \prec {}_M(\mathcal H_\Phi)_N\}
$$
which converges pointwise ultraweakly to $\Psi\in \ucp\varphi\rho MN$. We want to show that $\Psi \in U$. Fix $T=\sum_{p=1}^m a_p\otimes b_p^{op}\in M\odot N^{op}$ and a vector 
$\zeta=\sum_{j=1}^n z_j\odot \eta_j\in M\odot L^2(N,\rho)$. Then, we have  $\pi_\Theta(T)\zeta=\sum_{p=1}^m\sum_{j=1}^n a_p z_j\odot \eta_j b_p$
and thus,
$$
\|\zeta\|_\Theta^2
=
\sum_{i,j=1}^n
\langle \Theta(z_j^*z_i)\eta_i,\eta_j\rangle
\text{ and }
\|\pi_\Theta(T)\zeta\|_\Theta^2
=
\sum_{p,q=1}^m\sum_{i,j=1}^n
\left\langle
\Theta(z_j^*a_q^*a_p z_i)\,\eta_i b_p,\,
\eta_j b_q
\right\rangle.
$$
for any $\Theta\in \ucp\varphi\rho MN$. It is then immediate that if $\Psi_i\rightarrow \Psi$ in pointwise ultraweak topology, then we have $\|\zeta\|_{\Psi_i}^2\rightarrow \|\zeta\|_\Psi^2$ and $\|\pi_{\Psi_i}(T)\zeta\|_{\Psi_i}^2\rightarrow \|\pi_{\Psi}(T)\zeta\|_\Psi^2$.

Now, for each $i$, the assumption ${}_M\mathcal H_{\Psi_i}{}_N \prec {}_M\mathcal H_\Phi{}_N$ is equivalent to
$$\|\pi_{\mathcal H_{\Psi_i}}(T)\|\le \|\pi_{\mathcal H_\Phi}(T)\|.$$
Hence, for every vector $\zeta \in M\odot L^2(N,\rho)$, we have $\|\pi_{\Psi_i}(T)\zeta\|_{\Psi_i}\le \|\pi_{\mathcal H_\Phi}(T)\|\,\|\zeta\|_{\Psi_i}$. Passing to the limit, we have
$\|\pi_\Psi(T)\zeta\|_\Psi\le \|\pi_{\mathcal H_\Phi}(T)\|\,\|\zeta\|_\Psi
\,\forall\zeta\in M\odot L^2(N,\rho)$. Since the image of $M\odot L^2(N,\rho)$ is dense in $\mathcal H_\Psi$, it now follows that $$\|\pi_{\mathcal H_\Psi}(T)\|\le \|\pi_{\mathcal H_\Phi}(T)\|.$$
As $T\in M\odot N^{op}$ was arbitrary, this shows that $\Psi\in U$.\end{proof}

\section{The space of von Neumann algebras with state preserving conditional expectation}\label{sec:cond_exp}

In this section, we study the space of subalgebras of a given von-Neumann algebra $M$ equipped with a faithful normal state $\varphi$, which possess a $\varphi$-preserving conditional expectation. Our aim is to exhibit that this space has a natural Polish topology, which coincides with the restriction of the Effros-Maréchal topology on the space of all subalgebras of $M$.

Recall that, by Takesaki's Theorem, a von Neumann subalgebra $N\subset M$ has a $\varphi$-preserving conditional expectation $\mathbb{E}_N^\varphi:M\rightarrow N$ if and only if $\sigma_t^\varphi(N)=N$ for all $t\in\R$. Moreover, $\E^\varphi_N$ is the unique ucp map $M\rightarrow N$ such that $\varphi\circ \E_N^\varphi=\varphi$ and $\E_N^\varphi(x)=x$ for all $x\in N$. Such a map is automatically normal and is a $(\varphi,\varphi)$-Markov map. Note that $\E_N^\varphi$ is a projection onto $N$ and $N$ is the set of fixed points of $\E^\varphi_N$. The following Lemma is well known.

\begin{lem}\cite[Lemma 6.4]{bannonNoncommutativeJoinings2018} Let $M$ be a von Neumann algebra with a nfs $\varphi$. For any $\varphi$-preserving ucp map $\Phi\,:\, M\rightarrow M$, the set of fixed points
$$M^\Phi:=\{x\in M\,:\, \Phi(x)=x\}$$
is a von Neumann subalgebra.
\end{lem}

\begin{proof}
Note that, for all $x\in M^\Phi$, one has $x^*x=\Phi(x)^*\Phi(x)\leq\Phi(x^*x)$. Hence, $\Phi(x^*x)-x^*x\geq 0$. Applying the faithful state $\varphi$, we deduce that $\Phi(x^*x)=x^*x$. One shows similarly that $\Phi(xx^*)=xx^*$. Hence, $M^\Phi$ is a subset of the multiplicative domain of $\Phi$, which easily implies that $M^\Phi$ is a unital $*$-subalgebra. It is a von Neumann subalgebra of $M$ since $\Phi$ is normal, by Remark \ref{RmkNormal}.\end{proof}

In the sequel, a ucp map $\Phi\,:\,M\rightarrow M$ will be called \textbf{idempotent} whenever $\Phi^2:=\Phi\circ\Phi=\Phi$.

\begin{lem}\label{LemmaIMRN1}
Let $M$ be a von Neumann algebra equipped with a fixed faithful normal state $\varphi$. If $\Phi:M\rightarrow M$ be an idempodent $\varphi$-preserving ucp map then,
$$\Phi=\mathbb{E}_{M^\Phi}^\varphi.$$
\end{lem}

\begin{proof}
Since $\Phi$ is idempodent, we have $\Phi(M)= M^\Phi$. Hence $\Phi$ is a $\varphi$-preserving conditional expectation onto $M^\Phi$. The lemma follows by uniqueness $\mathbb{E}_{M^\Phi}^\varphi$.
\end{proof}
 
Let us now denote by $\mathcal{E}_\varphi\subset\mathcal{S}(M)$ the set of von Neumann subalgebras of $M$ that posses $\varphi$-preserving conditional expectations. By Lemma \ref{LemmaIMRN1} there is an obvious bijective correspondence between elements in $\mathcal{E}_{\varphi}$ and $\varphi$-preserving conditional expectations from $M$ onto its appropriate subalgebras. In the following Theorem, we view $\mathcal{E}_\varphi\subset \onelip_1(M_\varphi,M_\varphi)$.

Whenever $N\subseteq M$, we view $\LL^2(N,\varphi)\subseteq\LL^2(M,\varphi)$ via the isometric map $x\Omega\mapsto x\Omega_\varphi$, where $\Omega\in\LL^2(N,\varphi)$ and $\Omega_\varphi\in\LL^2(M,\varphi)$ are the cyclic vectors. Let $e_N\in\textbf{B}(\LL^2(M,\varphi))$ be the orthogonal projection onto $\LL^2(N,\varphi)$ and note that, if $N\in\mathcal{E}_\varphi$, then $e_N$ is given by $e_N(x\Omega_\varphi)=\mathbb{E}_N^\varphi(x)\Omega_\varphi$ and one has $\mathbb{E}_N^\varphi(x)=e_Nxe_N$, for all $x\in M$.

\begin{thm}\label{Subalgpolish}
Let $M$ be a von Neumann algebra equipped with a faithful normal state $\varphi$. The following holds.
\begin{enumerate}
\item $\mathcal{E}_\varphi$ is a closed subspace of $\onelip_1(M_\varphi,M_\varphi)$. In particular,
$\mathcal{E}_\varphi$ is a Polish space with topology induced from the topology of $\onelip_1(M_\varphi,M_\varphi)$.
\item This topology on $\mathcal{E}_\varphi$ is the smallest topology for which the maps
$$\mathcal{E}_\varphi\rightarrow [0,+\infty),\quad A\mapsto\Vert \mathbb{E}^\varphi_A(x)\Vert_{2,\varphi}$$
are continuous, for all $x\in M_1$.
\end{enumerate}
\end{thm}

\begin{proof}	
$(1)$. Note that the set $\mathcal{E}_{\varphi}$ is in particular a subspace of the space of conservative u.c.p. maps, which is Polish with respect to the topology of pointwise convergence in $\norm{\cdot}_{2,\varphi}$, as shown in Proposition \ref{ConservativeucpisPolish}. We show now that in fact the set $\mathcal{E}_{\varphi}$ is a closed subspace of space of conservative u.c.p. maps, whence it also follows that $\mathcal{E}_{\varphi}$ is closed in the space $\onelip_1(M_\varphi,M_\varphi)$ and thus that it is Polish. For this, note that if $\Phi,\Psi:M\rightarrow M$ are u.c.p. maps, then
\begin{align*}
\Phi^2(x)-\Psi^2(x)=\Phi(\Phi(x)-\Psi(x)) + (\Phi-\Psi)(\Psi(x)), \text{ }x\in M,
\end{align*}
which shows that $\Phi\mapsto \Phi^2$ is a continuous function from the space of conservative u.c.p. maps on $M$ to itself and thus, the conservative idempodents are the fixed points of this function, forming a closed set with respect to the topology of pointwise convergence in $\norm{\cdot}_{2,\varphi}$. By Lemma \ref{LemmaIMRN1} the proof is complete.

\vspace{0.2cm}

\noindent$(2)$. Let $\tau_1$ be the smallest topology on $\mathcal{E}_\varphi$ for which the maps $A\mapsto\Vert \mathbb{E}^\varphi_A(x)\Vert_{2,\varphi}$ are continuous, for all $x\in M_1$, and $\tau_2$ the topology on $\mathcal{E}_\varphi$ induced from $\onelip_1(M_\varphi,M_\varphi)$. By definition, the maps $A\mapsto\Vert \mathbb{E}^\varphi_A(x)\Vert_{2,\varphi}$ are $\tau_2$-continuous so $\tau_1\subseteq\tau_2$. Let us now show that $\tau_2\subseteq\tau_1$. It suffices to show that the maps $\mathcal{E}_\varphi\rightarrow (M_1,\Vert\cdot\Vert_{2,\varphi})$, $A\mapsto\mathbb{E}^\varphi_A(x)$ are $\tau_1$-continuous, for all $x\in M_1$.

Let $A_i\in\mathcal{E}_\varphi$ be a net $\tau_1$-converging to $A\in\mathcal{E}_\varphi$. Then, by the discussion before Theorem \ref{Subalgpolish}, $\Vert e_{A_i}(x\Omega_{\varphi})\Vert\rightarrow_i\Vert e_{A}(x\Omega_{\varphi})\Vert$ for all $x\in M_1$ hence, since $\overline{{\rm Span}}(M_1\Omega_\varphi)=\LL^2(M,\varphi)$, we have that $\Vert e_{A_i}\xi\Vert\rightarrow\Vert e_A\xi\Vert$ for all $\xi\in \LL^2(M,\varphi)$. It follows from Lemma \ref{lem: equiv cv for proj} that $e_{A_i}\rightarrow e_A$ strongly. Hence,
$$\Vert \mathbb{E}^\varphi_{A_i}(x)-\mathbb{E}^\varphi_{A}(x)\Vert_{2,\varphi}=\Vert e_{A_i}(x\Omega_\varphi)-e_A(x\Omega_\varphi)\Vert\rightarrow_i 0\text{ for all }x\in M.$$
It concludes the proof.\end{proof}

\begin{remark}\label{RmkGrassmanian}
Recall that the set $\mathcal G(\LL^2(M,\varphi))$ of closed subspaces of $\LL^2(M,\varphi)$ can be endowed with a natural Polish topology for which convergence of subspaces is equivalent to strong convergence of their orthogonal projections (by Riesz' Theorem and \cite[Poposition 4.6]{FLMMP24}). We note that the map $\mathcal E_\varphi\to \mathcal G(\LL^2(M,\varphi))$ which maps $N$ to $\LL^2(N,\varphi)$ is a homeomorphism onto its image. Indeed, since $\mathbb{E}_N^\varphi(x)\Omega_\varphi =e_N(x\Omega_\varphi)$ and because $M\Omega_\varphi\subset\LL^2(M,\varphi)$ is dense, the pointwise convergence in $2$-norm of $\mathbb E^\varphi_N$ is equivalent to the convergence of corresponding $e_N$ in the strong topology.
\end{remark}

Next, we show that on $\mathcal E_\varphi$, the topology of pointwise convergence in $2$-norm is in fact equivalent to topologies of pointwise convergence in several other topologies on $M$.

\begin{lem}\label{lem: pointwise in 2 norm equiv pointwise weak}
	On the space $\mathcal E_\varphi$, the following topologies coincide
	\begin{enumerate}
		\item  pointwise convergence in strong-$\ast$ topology on $M_1$.
		\item  pointwise convergence in strong topology on $M_1$.
		\item  pointwise convergence in $2$-norm on $M_1$.
		\item  pointwise convergence in weak topology on $M_1$.
		\item  pointwise convergence in $\LL^2$-weak topology on $M_1$.
        \item pointwise convergence in ultraweak topology on $M_1$.
	\end{enumerate}
\end{lem}

\begin{proof}
 It is straightforward to see that the first three topologies coincide as a special case of Proposition \ref{proposition: equivalent topologies on ucp}. By Lemma \ref{lem: topologies on ball are the same}, it follows that topologies $(4)$ and $(5)$ also coincide. So to show that topologies $(1)$ to $(5)$ coincide we only need to show that if $(N_n)$ is a sequence of elements of $\mathcal E_\varphi$ then $\mathbb E_{N_n}^\varphi\to \mathbb E_N^\varphi$ pointwise $\LL^2$-weakly if and only if $\mathbb E_{N_n}^\varphi\to \mathbb E_N^\varphi$ pointwise in $2$-norm. For this, we note that since the standard vacuum vector $\Omega_{\varphi}$ is cyclic for $M$ in $\LL^2(M,\varphi)$, the fact that $\mathbb E_{N_n}^\varphi(x)\to \mathbb E_N^\varphi(x)$ pointwise in $\LL^2$-weak topology implies that $e_{N_n}\xi\to e_N\xi$ $\LL^2$-weakly for all $\xi\in\LL^2(M,\varphi)$, equivalently that $e_{N_n}\to e_N$ weakly in $\textbf{B}(\LL^2(M,\varphi))$, where for $N\in\mathcal S(M)$, we denote by $e_N$ the orthogonal projection onto $\LL^2(N,\varphi)$.  But, since $e_{N_n}$ are orthogonal projections, we then have $e_{N_n}\to e_N$ strongly, so in particular we have $\mathbb E_{N_n}^\varphi(x)\to\mathbb E_N^\varphi(x)$ in $2$-norm for all $x\in M$, as desired. The converse is clear. Finally it is clear that pointwise $\LL^2$ convergence implies pointwise ultraweak which implies pointwise weak convergence.
\end{proof}

\begin{remark}
There is another way to prove that $\mathcal{E}_\varphi$ is Polish: it is not difficult to show that $\onelip_1(M_\varphi,M_\varphi)$ with the pointwise ultraweak topology is compact. Hence, it suffices to show that $\{\Phi\in\onelip_1(M_\varphi,M_\varphi):\Phi^2=\Phi\}$ is $G_\delta$ for the pointwise ultraweak topology. Since the map $\onelip_1(M_\varphi,M_\varphi)\times \onelip_1(M_\varphi,M_\varphi)\to \onelip_1(M_\varphi,M_\varphi)$ is separately continuous for the pointwise ultraweak, it is Baire class $1$. Hence $\Phi\to \Phi^2$ is Baire class $1$ for the pointwise ultraweak topology. This shows that $\mathcal{E}_\varphi$ is Polish for the induced pointwise ultraweak topology, which coincides with the usual topology on $\mathcal{E}_\varphi$, by Lemma \ref{lem: pointwise in 2 norm equiv pointwise weak}.
\end{remark}

We now reformulate a result of Tsukuda as rephrased by Haagerup and Winslow and provide a different proof.

\begin{thm}[{Tsukada \cite{tsukadaStrongLimitNeumann1985}, see also \cite[Remarks~2.11]{haagerupEffrosMarechalTopologySpace1998}}]\label{thm: marechal same as one lip}
	Let $M$ be a separable von Neuman algebra equipped with a faithful normal state $\varphi$. Then the space $\mathcal E_\varphi$ is a closed subset of $\mathcal S(M)$, and the induced topology coincides with the topology induced on it from $\onelip_1(M_\varphi,M_\varphi)$.
\end{thm}
\begin{proof} It is clear that, for any $\alpha\in{\rm Aut}(M)$, the map $\mathcal{S}(M)\rightarrow\mathcal{S}(M)$, $N\mapsto\alpha(N)$ is continuous. Since $\mathcal{E}_\varphi=\underset{t\in\R}{\bigcap}\{N\in\mathcal{S}(M)\,:\,\sigma_t^\varphi(N)=N\}$, we deduce that $\mathcal{E}_\varphi$ is closed.
	
Let us now show that the two topologies that we have coincide.

Assume that $N_n\to N$ in the Maréchal topology and let us show that for any $x\in M_1$, $\norm{\mathbb E^\varphi_{N_n}(x)-\mathbb E^\varphi_N(x)}_{2,\varphi}\to 0$.

By Lemma \ref{lem: pointwise in 2 norm equiv pointwise weak}, we only need to show that $\mathbb E^\varphi_{N_n}(x)\to \mathbb E^\varphi_N(x)$ weakly for all $x\in M_1$. By compactness of $M_1$ in the weak topology, this amounts to showing that $\mathbb E^\varphi_N(x)$ is the only accumulation point of the sequence $(\mathbb E^\varphi_{N_n}(x))_{n\in \mathbb{N}}$. For this, recall first that $\mathbb E^\varphi_N(x)$ is the unique element of $N$ which is the closest to $x$ in $2$-norm (it follows from the discussion preceding Theorem \ref{Subalgpolish}). Now we take an increasing sequence of integers $\phi(n)$ and $y\in M_1$ such that $\mathbb E^\varphi_{N_{\phi(n)}}(x)\to y$ weakly. If $y\notin N$ then there exists $\omega\in M_*$ such that $\omega\vert_N=0$ and $\omega(y)\neq 0$. However, $$\Vert\omega\vert_{N_{\phi(n)}}\Vert\geq\vert\omega(\E_{N_{\phi(n)}}^\varphi(x))\vert\rightarrow\vert\omega(y)\vert>0$$
and, since $N_{N_{\phi(n)}}\rightarrow N$, we have $\Vert\omega\vert_{N_{\phi(n)}}\Vert\rightarrow\Vert\omega\vert_N\Vert=0$. It implies that $y\in N$.

Hence, it remains to show that $\Vert x-y\Vert_{2,\varphi}\leq\Vert x-z\Vert_{2,\varphi}$ for all $z\in N$. Since $x-\mathbb{E}_{N_{\phi(n)}}^\varphi(x)\rightarrow x-y$ weakly, we have
$$\Vert x-y\Vert_{2,\varphi}\leq\limsup\Vert x-\E^\varphi_{N_{\phi(n)}}(x)\Vert_{2,\varphi}.$$

By \cite[Theorem 3.9]{haagerupEffrosMarechalTopologySpace1998} (see also \cite[Theorem 6.5]{FLMMP24}) there exists, for all $n\in\N$, a Maréchal to strong* continuous map $x_n\,:\,\mathcal{S}(M)\rightarrow M_1$ such that, for all $N\in\mathcal{S}(M)$, $\{x_n(N)\,:\,n\in\N\}$ is strong* dense in $N_1$. Now, fix $z\in N_1$ and $\epsilon>0$. Then, there exists $l\in\N$ such that $\Vert z-x_l(N)\Vert_{2,\varphi}<\epsilon$. Write, for all $n\in\N$,
\begin{eqnarray*}
\Vert x-\E^\varphi_{N_{\phi(n)}}(x)\Vert_{2,\varphi}
&\leq&\Vert x-x_l(N_{\phi(n)})\Vert_{2,\varphi}\\
&\leq&\Vert x-z\Vert_{2,\varphi}+\Vert z-x_l(N)\Vert_{2,\varphi}+\Vert x_l(N)-x_l(N_{\phi(n)})\Vert_{2,\varphi}\\
&<&\Vert x-z\Vert_{2,\varphi}+\epsilon+\Vert x_l(N)-x_l(N_{\phi(n)})\Vert_{2,\varphi}
\end{eqnarray*}
By continuity of $x_l$, we deduce that:
$$\Vert x-y\Vert_{2,\varphi}\leq\limsup\Vert x-\E^\varphi_{N_{\phi(n)}}(x)\Vert_{2,\varphi}\leq \Vert x-z\Vert_{2,\varphi}+\epsilon.$$
Since this holds for all $\epsilon>0$ and all $z\in N$, we deduce that $y=\mathbb{E}_N^\varphi(x)$.

Assume now that $N_n\rightarrow N$ in the topology induced from $\onelip_1(M_\varphi,M_\varphi)$. Let us show that $N_n\rightarrow N$ in the Maréchal topology. Let $\omega\in M_*$ and let us show that $\Vert\omega\vert_{N_n}\Vert\rightarrow\Vert\omega\vert_N\Vert$. We may and will assume that $\omega$ is positive. Then, viewing $M\subset\textbf{B}(H)$, where $H={\rm L}^2(M,\psi)$, there is a vector $\xi\in H$ such that $\omega(x)=\langle x\xi,\xi\rangle$ for all $x\in M$. Note that, for any $A\in\mathcal{E}_\varphi$, one has $\Vert\omega\vert_A\Vert_{A_*}=\Vert\omega\circ\mathbb{E}_A^\varphi\Vert_{M_*}$. Hence,
\begin{eqnarray*}
    \big\vert\Vert\omega\vert_N\Vert-\Vert\omega\vert_{N_n}\Vert\big\vert&=&\big\vert\Vert\omega\circ \mathbb{E}_N^\varphi\Vert-\Vert\omega\circ \mathbb{E}_{N_n}^\varphi\Vert\big\vert\leq\Vert\omega\circ\mathbb{E}_N^\varphi-\omega\circ\mathbb{E}_{N_n}^\varphi\Vert\\
    &=&\underset{x\in M_1}{\sup}\big\vert\langle \mathbb{E}_N^\varphi(x)\xi,\xi\rangle-\langle \mathbb{E}_{N_n}^\varphi(x)\xi,\xi\rangle\big\vert\\
    &=&\underset{x\in M_1}{\sup}\big\vert\langle xe_N\xi,e_N\xi\rangle-\langle xe_{N_n}\xi,e_{N_n}\xi\rangle\big\vert\\
    &\leq&\underset{x\in M_1}{\sup}\vert\langle x(e_N-e_{N_n})\xi,e_N\xi\rangle\vert+\underset{x\in M_1}{\sup}\vert\langle xe_{n_n}\xi,(e_N-e_{N_n})\xi\rangle\vert\\
    &\leq&2\Vert\xi\Vert\,\Vert e_N\xi-e_{N_n}\xi\Vert\rightarrow_n 0,
\end{eqnarray*}
where the convergences to $0$ follows from Remark \ref{RmkGrassmanian}.\end{proof}

\section{Topological properties of sets of subalgebras}\label{sec:subalgebras}

\subsection{Amenable subalgebras}

As is well known, following seminal work of Connes \cite{Co76}, the notion of amenability for von-Neumann algebras has several equivalent definitions. For our purpose, the following definition of semi-discreteness of von-Neumann algebras, which is equivalent to amenability, will prove useful.

\begin{df}A von Neumann algebra $M$ is \textbf{amenable} (or semidiscrete) if given $\eta\in M_\ast, \varepsilon >0$ and finitely many $x_1, x_2,...,x_k\in M_1$, there exist normal ucp maps $\Phi:M\rightarrow M_n(\mathbb{C})$ and $\Psi: M_n(\mathbb{C})\rightarrow M$ such that
$$|\eta(\Psi\circ\Phi(x_j)-x_j)|<\varepsilon\text{ for all }j\in \{1, 2,...,k\}.$$
\end{df}

In the following, we study the set of amenable subalgebras having a $\varphi$-preserving conditional expectation and show that it is closed in $\mathcal{E}_\varphi$. This generalises a result of \cite{TYH25} from the tracial case with a different and simpler proof. 
It implies that it is not possible to approximate a non-amenable Connes-embeddable ${\rm II}_1$ factor $N$ (eg $N=\LL(\mathbb{F}_n)$, $n\geq 2$) ``from inside'', i.e. using amenable subalgebras of $N$.
This is in sharp contrast to the result Haagerup-Winslow \cite[Theorem 5.8]{haagerupEffrosMarechalTopology2000}: they show that every Connes-embeddable ${\rm II}_1$-factor $N\subset\textbf{B}(H)$ is is approximable in the Maréchal topology on $\textbf{B}(H)$ by amenable von-Neumann algebras.

\begin{prop}\label{prop: amenable is closed v1}
 Let $M$ be a von-Neumann algebra with a nfs $\varphi$. The set of amenable subalgebras in $\mathcal{E}_\varphi$ is closed in $\mathcal{E}_\varphi$.  
\end{prop}

\begin{proof}
Let  $(N_n)_n$ be a sequence in $\mathcal{E}_\varphi$ converging to $N\in\mathcal{E}_\varphi$ and suppose that each $N_n$ is amenable. Let $x_1, x_2,...,x_k\in N, \eta\in N_\ast, \varepsilon>0$. Since $N_n\rightarrow N$, we have that $\mathbb{E}^\varphi_{N_n}(x_j)\rightarrow\mathbb{E}_N^\varphi(x_j)=x_j$, weakly for all $j$. Hence, we can find $i$ such that $|\eta\circ \mathbb E_N^\varphi(\mathbb{E}^\varphi_{N_i}(x_j)-x_j)|<\varepsilon/2$ for all $j$.  Now, since $N_i$ is amenable, hence for $\mathbb{E}^\varphi_{N_i}(x_1),...,\mathbb{E}^\varphi_{N_i}(x_k)\in N_i, \eta\circ \mathbb{E}^\varphi_N\vert_{N_i}, \varepsilon/2$, we can find normal ucp maps $\Phi:N_i\rightarrow M_n(\mathbb{C})$ and $\Psi:M_n(\mathbb{C})\rightarrow N_i$ such that 
 $$|\eta\circ \mathbb{E}_N^\varphi(\Psi\circ\Phi(\mathbb{E}^\varphi_{N_i}(x_j))-\mathbb{E}^\varphi_{N_i}(x_j))|<\varepsilon/2$$
 for all $j$. Let us now consider the normal ucp maps $\Phi_0=\Phi\circ \mathbb{E}^\varphi_{N_i}:N\rightarrow M_n(\mathbb{C})$ and $\Psi_0=\mathbb{E}^\varphi_{N}\circ\Psi : M_n(\mathbb{C})\rightarrow N$. It is then easy to see that
 $$|\eta(\Psi_0\circ\Phi_0(x_j)-x_j)|<\varepsilon\text{ for all }j.$$ 
 It completes the proof.\end{proof}

This result can in fact be further generalised. Assume that $N\subseteq M$ is a von Neumann subalgebra of $M$, equipped with a faithful normal state $\phi$ and with $\mathbb{E}_N^{\phi}$ denoting the $\phi$ preserving conditional expectation onto $N$. Recall the following definition \cite{Is19, OzawaPopa}.

\begin{df}
    We say that $M$ is \textbf{amenable relative to $N$} if there exists a conditional expectation $\theta:\langle M,N\rangle\to M$ such that $\theta|_M=id_M$.

\end{df}

\begin{remark}
Let $P\subseteq N\subseteq M$, with fixed faithful $\phi$-invariant conditional expectations from $M$ onto $N$ and $P$ denoted $\mathbb{E}_N^\phi$ and $\mathbb{E}_P^\phi$ respectively, the notion of relative amenability of $N$ relative to $P$ in $M$ is given by the existence a conditional expectation $\Theta:\langle M,P\rangle\to N$ such that $\Theta|_M=\mathbb{E}_N^\phi$ \cite{Is19}, denoted as $N\lessdot_M P$. However, in this case, it can be shown that $N$ is amenable relative to $P$ if and only if $N$ is amenable relative to $P$ in $M$. To see this, assume first that $N$ is amenable relative to $P$. Thus, there is a conditional expectation $\theta:\langle N,P\rangle\to N$ such that $\theta|_N=id_N$. Let $e_N:\LL^2(M,\phi)\to \LL^2(N,\phi)$ be the Jones projection associated with $\E^\phi_N$. Define the UCP map  $\Gamma_N:\langle M,P\rangle\to \langle N,P\rangle$
by $\Gamma_N(T)=e_NTe_N\big|_{L^2(N,\phi)}$. It is easy to see that on the algebraic span of $M e_P M$, we have $\Gamma_N(xe_Py)=\mathbb{E}^\phi_N(x)e_P \mathbb{E}^\phi_N(y),
x,y\in M$. Indeed, let $z\in N$. Then $e_Nxe_Pyz\widehat 1
= e_Nx\mathbb{E}^\phi_P(yz)\widehat{1}
= \mathbb{E}^\phi_N(x\mathbb{E}^\phi_P(yz))\widehat{1}$. Since $\mathbb{E}^\phi_P(yz)\in P\subseteq N$, this equals $\mathbb{E}^\phi_N(x)\mathbb{E}^\phi_P(yz)\widehat{1}$. On the other hand, $\mathbb{E}^\phi_N(x)e_P\mathbb{E}^\phi_N(y)z\widehat 1= \mathbb{E}^\phi_N(x)\mathbb{E}^\phi_P(\mathbb{E}^\phi_N(y)z)\widehat{1}=\mathbb{E}^\phi_N(x)\mathbb{E}^\phi_P(yz)\widehat{1}$, since $z\in N$
and since $\mathbb{E}^\phi_P\circ \mathbb{E}^\phi_N=\mathbb{E}^\phi_P$. Thus, we have
$$
e_Nxe_Py e_N
=
\mathbb{E}^\phi_N(x)e_P\mathbb{E}^\phi_N(y)
$$
on $\LL^2(N,\phi)$.

In particular, for $x\in M$, $\Gamma_N(x)=\mathbb{E}^\phi_N(x)$. Now, we define
$$
\Phi:=\theta\circ \Gamma_N:\langle M,P\rangle\to N.
$$
Then $\Phi$ is UCP, its range is contained in $N$, and for $x\in M$, $\Phi(x)=\theta(\mathbb{E}^\phi_N(x))=\mathbb{E}^\phi_N(x)$. Also, if $n\in N$, then $\Phi(n)=\mathbb{E}^\phi_N(n)=n$. Hence $\Phi$ is a norm-one projection from $\langle M,P\rangle$ onto $N$,
so it is a conditional expectation, and $\Phi|_M=\mathbb{E}^\phi_N$. Thus $(N,E_N)\lessdot_M P$.

Conversely, assume $(N,E_N)\lessdot_M P$. Thus there is a conditional expectation $\Phi:\langle M,P\rangle\to N$ such that $\Phi|_M=\mathbb{E}^\phi_N$. Since $P\subseteq N\subseteq M$, we have a natural inclusion $\langle N,P\rangle\subseteq \langle M,P\rangle$. Restricting $\Phi$ to $\langle N,P\rangle$, we note that for every $n\in N,\Phi(n)=\mathbb{E}^\phi_N(n)=n$. Therefore
$$
\Phi|_{\langle N,P\rangle}:\langle N,P\rangle\to N
$$
is a conditional expectation fixing $N$ pointwise, thus $N$ is amenable relative to $P$.
\end{remark}

\begin{thm}\label{thm:rel-amenable-closed-core-proof}
Let $M$ be a von Neumann algebra equipped with a faithful normal state $\phi$.
Let $P\in \mathcal E_\phi$. Then the set
$$
\mathcal A_P
:=
\{\,N\in \mathcal E_\phi
\mid
P\subseteq N\subseteq M
\text{ and } N \text{ is amenable relative to } P\,\}
$$
is closed in $\mathcal E_\phi$. Since $\mathcal E_\phi$ is closed in $\mathcal{S}(M)$, it is
also closed as a subset of $\mathcal S (M)$.
\end{thm}

\begin{proof}

First note that since $\E_Q^\phi$ is $\phi$-preserving, Takesaki's theorem gives $\sigma_t^\phi(Q)=Q$ and $\E_Q^\phi\circ\sigma_t^\phi=\sigma_t^\phi\circ \E_Q^\phi$ for all $t\in\mathbb R$. Hence the crossed product $\widetilde Q=Q\rtimes_{\sigma^\phi|_Q}\mathbb R$
is naturally represented as a von Neumann subalgebra of $\widetilde M=M\rtimes_{\sigma^\phi}\mathbb R$. Moreover, in the standard crossed-product representation, the orthogonal
projection $1\otimes e_Q:
\LL^2(\mathbb R,\LL^2(M,\phi))
\to
\LL^2(\mathbb R,\LL^2(Q,\phi|_Q))$ implements a faithful normal conditional expectation $\widetilde \E_Q:\widetilde M\to\widetilde Q$
characterized by $\widetilde{\E}_Q(x\lambda_t)=\E_Q^\phi(x)\lambda_t$ for all $x\in M$ and $t\in\R$.

Let now $(N_k)_k$ be a sequence in $\mathcal A_P$ such that $N_k\to N$ in $\mathcal E_\phi$. We want to show that $N\in\mathcal A_P$. To that end, first note that it is immediate that that $P\subseteq N$.

Next, by \cite{haagerupEffrosMarechalTopologySpace1998} (Theorem~6.19), we have continuity for continuous cores, so the convergence $N_k\to N$ in the Effros--Mar\'echal topology implies $\widetilde N_k=c_\phi(N_k)\to c_\phi(N)=\widetilde N$ in the Effros--Mar\'echal topology. We now choose a faithful normal state $\chi$ on $\widetilde P= c_\phi(P)$ and define $\omega:=\chi\circ \widetilde \E_P$ on $\widetilde M$. Since $P\subseteq N_k$ for all $k$, and $P\subseteq N$, we have $\E^\phi_P\circ \E^\phi_{N_k} = \E^\phi_P$ and $ \E^\phi_P\circ \E^\phi_N
= \E^\phi_P$ hence, $\widetilde \E_P\circ \widetilde \E_{N_k} = \widetilde \E_P$ and $\widetilde \E_P\circ \widetilde \E_N
= \widetilde \E_P$. Hence, $\omega\circ \widetilde \E_{N_k}=\omega$ and $\omega\circ \widetilde \E_N=\omega$. Thus $\widetilde N_k,\widetilde N\in\mathcal E_\omega(\widetilde M)$. Since $\widetilde N_k\to\widetilde N$, thus we have, by Theorem \ref{thm: marechal same as one lip}, that $\widetilde \E_{N_i}(X)\to \widetilde \E_N(X)$ in $\|\cdot\|_{2,\omega}$, hence ultraweakly, for every $X\in\widetilde M$.

Now, for any $k
\in \mathbb{N}$ and since $N_k\in\mathcal A_P$, $N_k$ is amenable relative to $P$. So, we have $(N_k,E_{N_k})\lessdot_M P$. By \cite{Is19} (Theorem~3.2), we have $(\widetilde N_k,\widetilde E_{N_k})\lessdot_{\widetilde M}\widetilde P$. So, there exists a conditional
expectation $\langle \widetilde N_k,\widetilde P\rangle\to \widetilde N_k$ fixing $\widetilde N_k$ pointwise, by the previous remark. Since $\widetilde N_k$ is semifinite, so by \cite{Is19} (Appendix Theorem~A.6), relative amenability of $\widetilde N_k$ with respect to $\widetilde P$ (inside the ambient algebra $\widetilde N_k$), implies 
$$
{}_{\widetilde N_k}\LL^2(\widetilde N_k)_{\widetilde N_k}
\prec
{}_{\widetilde N_k}
\big(
\LL^2(\widetilde N_k)\otimes_{\widetilde P}\LL^2(\widetilde N_k)
\big)_{\widetilde N_k}.
$$

Note now that Connes tensor product with fixed correspondences preserves weak containment (follows from \cite[Prop 2.2.1]{PoCorrespondences}). Therefore, tensoring the preceding weak containment on the left with ${}_{\widetilde M}\LL^2(\widetilde M)_{\widetilde N_k}$ and on the right with ${}_{\widetilde N_k}\LL^2(\widetilde M)_{\widetilde M}$, we obtain ${}_{\widetilde M}
\big(
\LL^2(\widetilde M)\otimes_{\widetilde N_k}\LL^2(\widetilde M)
\big)_{\widetilde M}
\prec
{}_{\widetilde M}
\big(
\LL^2(\widetilde M)\otimes_{\widetilde P}\LL^2(\widetilde M)
\big)_{\widetilde M}$. Here, we use that $\LL^2(\widetilde M)\otimes_{\widetilde N_k}
\LL^2(\widetilde N_k)\otimes_{\widetilde N_k}
\LL^2(\widetilde M)
\cong
\LL^2(\widetilde M)\otimes_{\widetilde N_k}\LL^2(\widetilde M)$ and that $\LL^2(\widetilde M)\otimes_{\widetilde N_k}
\big(
\LL^2(\widetilde N_k)\otimes_{\widetilde P}\LL^2(\widetilde N_k)
\big)
\otimes_{\widetilde N_k}
\LL^2(\widetilde M)
\cong
\LL^2(\widetilde M)\otimes_{\widetilde P}\LL^2(\widetilde M)$. Thus, equivalently, for every $k$, we have $\mathcal H_{\widetilde \E_{N_k}}\prec \mathcal H_{\widetilde \E_P}$ as $\widetilde M$-$\widetilde M$ correspondences. By Proposition \ref{prop: weak-containment-closed}, for fixed $\widetilde \E_P$ the set
$$
\left\{
\Psi\in {\rm UCP}_{\omega,\omega}(\widetilde M,\widetilde M)
\mid
{}_{\widetilde M}\mathcal H_\Psi{}_{\widetilde M}
\prec
{}_{\widetilde M}\mathcal H_{\widetilde \E_P}{}_{\widetilde M}
\right\}
$$
is closed under pointwise ultraweak convergence. Since each $\widetilde \E_{N_k}$ belongs to this set, thus the limit $\widetilde \E_N$ also
belongs to this set. Therefore, we have ${}_{\widetilde M}\mathcal H_{\widetilde \E_N}{}_{\widetilde M}
\prec
{}_{\widetilde M}\mathcal H_{\widetilde \E_P}{}_{\widetilde M}$.

Restricting the right action to $\widetilde N$ gives
$$
{}_{\widetilde M}
\big(
\LL^2(\widetilde M)\otimes_{\widetilde N}\LL^2(\widetilde M)
\big)_{\widetilde N}
\prec
{}_{\widetilde M}
\big(
\LL^2(\widetilde M)\otimes_{\widetilde P}\LL^2(\widetilde M)
\big)_{\widetilde N}.
$$
which implies that ${}_{\widetilde M}\LL^2(\widetilde M)_{\widetilde N}
\prec
{}_{\widetilde M}
\LL^2(\widetilde M)\otimes_{\widetilde P}\LL^2(\widetilde M)_{\widetilde N}$, using ${}_{\widetilde M}\LL^2(\widetilde M)_{\widetilde N}\cong {}_{\widetilde M}
\LL^2(\widetilde M)\otimes_{\widetilde N}\LL^2(\widetilde N)_{\widetilde N}
\subseteq
{}_{\widetilde M}
\LL^2(\widetilde M)\otimes_{\widetilde N}\LL^2(\widetilde M)_{\widetilde N}$ which follows from \cite[Lemma 1.3.3]{PoCorrespondences}. Now, by \cite{Is19} (Theorem A.6 (condition $(1)$ implies condition $(3)$)), we obtain a UCP map $\Psi:
\langle \widetilde M,\widetilde P\rangle
\to
\langle \widetilde M,\widetilde N\rangle$ which fixes $\widetilde M$ pointwise. Thus, it follows now from \cite{Is19}
(Theorem~3.2) that $(\widetilde N,\widetilde \E_N)\lessdot_{\widetilde M}\widetilde P$.

Finally, again by \cite{Is19} (Theorem~3.2), the core condition
$(\widetilde N,\widetilde \E_N)\lessdot_{\widetilde M}\widetilde P$ is equivalent to $(N,\E^\phi_N)\lessdot_M P$, which gives that $N\in\mathcal A_P$. Thus, $\mathcal A_P$ is closed in $\mathcal E_\phi$, thus in $\mathcal{S}(M)$ as well.\end{proof}

 \subsection{Subalgebras with the Haagerup approximation property}
We now study the space of subalgebras with expectation and possessing the (relative) Haagerup approximation property. The Haagerup property has been named after the seminal work of Haagerup \cite{Ha78}. We follow here the formulation of the relative Haagerup property for arbitrary von Neumann algebras given in \cite{CKSVW23}.

\medskip

\noindent Let $M$ be a von Neumann algebra with a nfs $\varphi$. Let $\Phi\,:\,M\rightarrow M$ be a normal ucp map and assume that $\varphi\circ\Phi\leq\varphi$ so that the map $\widehat{\Phi}\,:\,\LL^2(M,\varphi)\rightarrow\LL^2(M,\varphi)$, $x\Omega_\varphi\mapsto\Phi(x)\Omega_\varphi$ is bounded (actually a contraction).

\begin{df}[\cite{CKSVW23}] 
Let $\varphi$ be a nfs on $M$ and $N,P\in\mathcal{E}_\varphi$ such that $P\subseteq N$. We say that $M$ has the \textbf{Haagerup property relative to $N$} if for all $\varepsilon >0$, $k\geq 1$, $x_1,x_2,...,x_k\in M_1$, there exists a normal ucp $N$-$N$-bimodular map $\Phi:M\rightarrow M$ s.t.

\begin{enumerate}
\item $\varphi\circ\Phi\leq\varphi$ and $\widehat{\Phi}\in \mathcal{K}_\varphi(M,N)$,
\item $\Vert\Phi(x_j)-x_j\Vert_{2,\varphi}<\varepsilon$ for all $j$,
\end{enumerate}
where $\mathcal{K}_\varphi(M,N):=\overline{{\rm Span}}^{\Vert\cdot\Vert}\{xe_Ny:x,y\in M\}\subset\textbf{B}(\LL^2(M,\varphi))$.
\end{df}

We now show that the set of von Neumann subalgebras of $M$ which are images of $\varphi$ preserving conditional expectations and have Haagerup property relative to some $P\in \mathcal{E}_\varphi$ is closed. This generalises a result of Jolissaint-Valette \cite{jolisaintvaletteHAP} on closedness of subalgebras of a finite von Neumann algebra having the Haagerup property.

\begin{thm}\label{thm: rel haagerup is closed}
Let $M$ be a von Neumann algebra equipped with a faithful normal state $\varphi$,
and let $P\in \mathcal E_\varphi$. Then the set
$$
\mathcal H_P
:=
\{\,N\in \mathcal E_\varphi \mid P\subseteq N
\text{ and } N \text{ has the Haagerup property relative to } P\,\}
$$
is closed in $\mathcal E_\varphi$.
\end{thm}

\begin{proof}
 Let $(N_n)_n$ be a sequence in $\mathcal{H}_P$ and $N\in\mathcal{E}_\varphi$ and assume that $N_n\rightarrow N$ so that $P\subseteq N$. Let $x_1,x_2,...,x_k\in N_1$ and $\varepsilon >0$. Let $i$ be such that $\|\mathbb{E}^\varphi_{N_i}(x_j)-x_j\|_2 <\varepsilon/2$ for all $j$. Since $N_i\in\mathcal{H}_P$, there exist a normal ucp $P$-$P$-bimodular map $\Phi:N_i\rightarrow N_i$ such that $\varphi\circ \Phi\leq \varphi$, $\widehat{\Phi}\in\mathcal{K}_\varphi(N_i,P)$ and $\|\Phi(\mathbb{E}_{N_i}(x_j))-\mathbb{E}_{N_i}(x_j)\|_{2,\varphi}<\varepsilon/2$ for all $j$. Define the normal ucp map $\Psi=\mathbb{E}_N^\varphi\circ\Phi\circ\mathbb{E}^\varphi_{N_i}\vert_{N}: N\rightarrow N$. Since $\mathbb{E}_N^\varphi$ and $\mathbb{E}_{N_i}^\varphi$ are $\varphi$-preserving, we have $\varphi\circ \Psi\leq \varphi$ and since $P\subset N,N_i$, the map $\Psi$ is $P$-$P$-bimodular. Moreover, 
 \begin{eqnarray*}
 \|\Psi(x_j)-x_j\|_{2,\varphi}&=& \|\mathbb{E}_N^\varphi(\Phi(\mathbb{E}_{N_i}^\varphi(x_j)))-x_j\|_{2,\varphi} \\
&\leq& \|\mathbb{E}_N^\varphi(\Phi(\mathbb{E}_{N_i}^\varphi(x_j)))-\mathbb{E}_{N}^\varphi(\mathbb{E}_{N_i}^\varphi(x_j))\|_2+\|\mathbb{E}_{N}^\varphi(\mathbb{E}_{N_i}^\varphi(x_j))-x_j\|_2<\varepsilon
 \end{eqnarray*}
It remains to show that $\widehat{\Psi}\in\mathcal{K}_\varphi(N,P)$. By the discussion before Theorem \ref{Subalgpolish}, one has ${\widehat{\mathbb{E}}^\varphi_A}=e_A$, the orthogonal projection onto ${\rm L}^2(A,\varphi)\subset\LL^2(M,\varphi)$ for any $A\in\mathcal{E}_\varphi$. It follows that $\widehat{\Psi}=e_N\widehat{\Phi}e_{N_i}\vert_{{\rm L}^2(N)}$. Let $\delta>0$. Since $\widehat{\Phi}\in\mathcal{K}_\varphi(N_i,P)$, there exists $l$ and $a_1,\dots,a_l,b_1,\dots,b_l\in N_i$ such that, with $s:=\sum_{k=1}^la_ke_Pb_k\in\mathcal{K}_\varphi(N_i,P)$ one has $\Vert \widehat{\Phi}-s\Vert<\delta$. Note that, for all $x\in N$, denoting by $\Omega$ the cyclic vector,
\begin{eqnarray*}
e_Nse_{N_i}(x\Omega)&=&\sum_{k=1}^le_Na_ke_Pb_ke_{N_i}(x\Omega)=\sum_k\E^\varphi_N(a_k\E^\varphi_P(b_k\E^\varphi_{N_i}(x)))\Omega\\
&=&\sum_k\E^\varphi_N(a_k\E^\varphi_P(\E^\varphi_{N_i}(b_kx)))\Omega
=\sum_k\E^\varphi_N(a_k\E^\varphi_P(b_kx))\Omega\\
&=&\sum_k\E^\varphi_N(a_k)\E^\varphi_P(b_kx)\Omega=\sum_k\E^\varphi_N(a_k)\E^\varphi_P(\E_N^\varphi(b_k)x)\Omega\\
&=&\sum_k\E^\varphi_N(a_k)e_P\E_N^\varphi(b_k)(x\Omega)
\end{eqnarray*}
Hence, $t:=e_Nse_{N_i}\vert_{{\rm L}^2(N)}=\sum_k\E^\varphi_N(a_k)e_P\E_N^\varphi(b_k)\in\mathcal{K}_\varphi(N,P)$. Moreover,
$$\Vert\widehat{\Psi}-t\Vert=\Vert e_N\widehat{\Phi}e_{N_i}\vert_{\LL^2(N)}-e_Nse_{N_i}\vert_{{\rm L}^2(N)}\Vert\leq\Vert\widehat{\Phi}-s\Vert<\delta.$$
It follows that $\widehat{\Psi}\in\mathcal{K}_\varphi(N,P)$.\end{proof}

\begin{remark}
As observed in the amenable case, it is not true, in general, that the set of subalgebra with the Haagerup property is closed in $\mathcal{S}(M)$. Actually, for any infinite dimensional separable Hilbert space $H$, the set of subalgebras of $\textbf{B}(H)$ with the Haagerup property is not closed in $\mathcal{S}(\textbf{B}(H))$. Indeed, take any Connes-embeddable ${\rm II}_1$-factor $N\subset\textbf{B}(H)$ without the Haagerup property (eg. $N=\LL(\Gamma)\subset\textbf{B}(l^2(\Gamma))$ with $\Gamma$ an icc property $(T)$ group). By \cite[Theorem 5.8]{haagerupEffrosMarechalTopology2000} $N$ is approximable by amenable von neumann subalgebras in $\textbf{B}(H)$.
\end{remark}

\subsection{Weakly amenable subalgebras}\label{sec: weakly amenable} Recall that, given a linear map $\Phi\,:\, M\rightarrow M$ its completely bounded norm is defined by
$$\Vert\Phi\Vert_{cb}:=\underset{n\geq 1}{\sup}\Vert\Phi_n\Vert,$$
where $\Phi_n:={\rm id}\ot\Phi\,:\, M_n(\C)\ot M\rightarrow M_n(\C)\ot M$. Note that if $\Phi$ is ucp then $\Vert\Phi\Vert_{cb}=1$. Weak amenability has been introduced by Haagerup.

\begin{df}
Let $\varphi$ be a nfs on $M$. We say that $M$ is \textbf{weakly amenable} if there exists $C>0$ such that for all $k\geq 1$, all $x_1,x_2,...,x_k\in M_1$ and all $\varepsilon >0$, there exists a $\sigma$-weakly continuous linear finite rank map $\Phi:M\rightarrow M$ such that:
\begin{itemize}
\item $\|\Phi(x_j)-x_j\|_{2,\varphi}<\varepsilon$ for all $j$,
\item $\|\Phi\|_{cb}\leq C$.
\end{itemize}
\end{df}

Given a von-Neumann algebra $M$, the infimum over the set of all $C>0$ for which the statement above holds, is called the \textbf{Cowling-Haagerup constant of $M$} and is denoted by $\Lambda_{cb}(M)$. This constant actually does not depend of the choice of a nfs $\varphi$ on $M$.

\begin{prop}\label{thm:weak-amenability-closed}
The map $\Lambda_{cb}:\mathcal{E}_\varphi\rightarrow [0,+\infty]$, $N\mapsto \Lambda_{cb}(N)$
 is lower semicontinuous. In particular, the set of weakly amenable von Neumann subalgebra of $M$ with $\varphi$-preserving expectation is $F_\sigma$ in $\mathcal{E}_\varphi$.
\end{prop}

\begin{proof}
Fix $C>0$. It suffices to show that $\{N\in \mathcal{E}_\varphi\,:\, \Lambda_{cb}(N)\leq C\}$ is closed. Let $N_n,N\in\mathcal{E}_\varphi$ such that $N_n\rightarrow N$ and assume that $\Lambda_{cb}(N)\leq C$ for all $n$. Let $x_1,x_2,...,x_k\in N_1$ and $\varepsilon >0$. Let $i$ be such that $\|\mathbb{E}^\varphi_{N_i}(x_j))-x_j\|_{2,\varphi} \leq \varepsilon/2$ for all $j=1,\dots,k$. Now consider the elements $\mathbb{E}^\varphi_{N_i}(x_j)$ in $N_i$. Since $\Lambda_{cb}(N_i)\leq C$, there exists a $\sigma$-weakly continuous linear finite rank map $\Phi:N_i\rightarrow N_i$ such that $\|\Phi(\mathbb{E}^\varphi_{N_i}(x_j))-\mathbb{E}^\varphi_{N_i}(x_j)\|_{2,\varphi}<\varepsilon/2$ for all $j$ and further, $\|\Phi\|_{cb} < C+\delta$, with $\delta>0$ chosen arbitrarily. Consider now the map $\Psi=\mathbb{E}^\varphi_{N}\circ\Phi\circ \mathbb{E}^\varphi_{N_i}\vert_{N}:N\rightarrow N$. Note that $\Psi$ is $\sigma$-weakly continuous, linear of finite rank and $\|\Psi\|_{cb}\leq\|\Phi\|_{cb} < C+\delta$. Also, 
 $$\|\Psi(x_j)-x_j\|_{2,\varphi}\leq \|\mathbb{E}^\varphi_{N}(\Phi(\mathbb{E}^\varphi_{N_i}(x_j)))-\mathbb{E}^\varphi_{N}(\mathbb{E}^\varphi_{N_i}(x_j))\|_2+\|\mathbb{E}^\varphi_{N}(\mathbb{E}^\varphi_{N_i}(x_j))-x_j\|_2 <\varepsilon$$
 for all $j$. Since our choice of $\delta$ was arbitrary, it hence follows that $\Lambda_{cb}(N)\leq C$.\end{proof}

 \begin{remark}
It is not true, in general, that the map $\Lambda_{cb}\,:\,\mathcal{S}(M)\rightarrow[0,+\infty]$ is lower semi-continuous. Actually it is not lower semi-contiunous for $M=\textbf{B}(H)$ ($H$ any infinite dimensional separable Hilbert space). Indeed, by \cite[Theorem 5.8]{haagerupEffrosMarechalTopology2000}, $\{N\in\mathcal{S}(\textbf{B}(H))\,:\,\Lambda_{cb}(N)\leq 1\}$ is not closed since it contains all the amenable subalgebras of $\textbf{B}(H)$ but does not contain any Connes-embeddable and not weakly amenable subalgebra such as $\LL(\Z^2\rtimes{\rm SL}_2(\Z))$ which is not weakly amenable by is a result of Haagerup \cite{Ha16} and Connes-embeddable since $\Z^2\rtimes{\rm SL}_2(\Z)$ is sofic, by \cite{ES11} and because $\Z^2\rtimes{\rm SL}_2(\Z)\simeq \Z^2\rtimes\Z/4\Z\underset{\Z^2\rtimes\Z/2\Z}{*}\Z^2\rtimes\Z/6\Z$. 
 \end{remark}

 We now provide examples showing that the $\Lambda_{cb}$ map $\mathcal{E}_\varphi\rightarrow[0,+\infty]$, $N\mapsto \Lambda_{cb}(N)$ is not continuous in general and the set of weakly amenable subalgebras in $\mathcal{E}_\varphi$ is neither closed or open in general. We need some preparation.

 Let ${\rm Sub}(\Gamma)$ be the set of subgroups of a given discrete countable group $\Gamma$ and view ${\rm Sub}(\Gamma)\subset \{0,1\}^\Gamma$ by identifying a subgroup $\Lambda$ with its characteristic function $1_\Lambda\in\{0,1\}^\Gamma$. It is easy to see that ${\rm Sub}(\Gamma)$ is closed in $\{0,1\}^\Gamma$ hence it is a compact space (with the induced product topology). The following is well know \cite[Proposition 4.1]{TYH25} but we provide a simpler proof.

 \begin{prop}\label{PropGroupsAlgebras}
  The map ${\rm Sub}(\Gamma)\rightarrow\mathcal{S}(\LL(\Gamma))$, $\Lambda\mapsto{\rm L}(\Lambda)$ is continuous.   
 \end{prop}

 \begin{proof}
 Let $\tau$ be the canonical trace on $\LL(\Gamma)$. By Lemma \ref{lem: pointwise in 2 norm equiv pointwise weak}, it suffices to show that, for all $x\in \LL(\Gamma)$, the map ${\rm Sub}(\Gamma)\rightarrow \LL^2(M)$, $\Lambda\mapsto\mathbb{E}_{\LL(\Lambda)}(x)$ is continuous, where $\mathbb{E}_N$ denotes the unique $\tau$-preserving conditional expectation onto $N\subset\LL(\Gamma)$.

 Let $X:=\{x\in \LL(\Gamma)\,:\,\Lambda\mapsto\mathbb{E}_{\LL(\Lambda)}(x)\text{ is }\LL^2\text{ continuous}\}$. Then $X$ is clearly a subspace of $\LL(\Gamma)$ and contains $\{\lambda_\gamma\,:\,\gamma\in\Gamma\}$ since $\mathbb{E}_{\LL(\Lambda)}(\lambda_\gamma)=1_\Lambda(\gamma)\lambda_\gamma$. Hence, it suffices to show that $X$ is $\Vert\cdot\Vert_2$ closed. Let $x_n\in X$ be a sequence and $x\in\LL(\Gamma)$ such that $\Vert x_n-x\Vert_2\rightarrow 0$. Since, for all $\Lambda\in{\rm Sub}(\Gamma)$,
 $$\Vert\mathbb{E}_{\LL(\Lambda)}(x)-\mathbb{E}_{\LL(\Lambda)}(x_n)\Vert_2\leq\Vert x_n-x\Vert_2,$$
 it follows that the sequence of continuous functions $(\Lambda\mapsto \mathbb{E}_{\LL(\Lambda)}(x_n))$ converges uniformly on ${\rm Sub}(\Gamma)$ to the function $(\Lambda\mapsto \mathbb{E}_{\LL(\Lambda)}(x))$ hence, $x\in X$.\end{proof}

 In the sequel we write $\Lambda_{cb}(\Gamma):=\Lambda_{cb}(\LL(\Gamma))$.

 \begin{remark}\label{RmkNotWA}
Let $\Gamma$ be a countable discrete group. From Proposition \ref{PropGroupsAlgebras} we deduce the following.
\begin{itemize}
    \item If ${\rm Sub}(\Gamma)\rightarrow[0,+\infty]$, $G\mapsto\Lambda_{cb}(G)$ is not continuous then $\mathcal{S}(\LL(\Gamma))\rightarrow [0,+\infty]$, $N\mapsto \Lambda_{cb}(N)$ is not continuous.
    \item If the set of weakly amenable subgroups of $\Gamma$ is not closed (resp. open) in $\LL(\Gamma)$ then the set of weakly amenable von Neumann subalgebras of $\LL(\Gamma)$ is not closed (resp. open) in $\mathcal{S}(\LL(\Gamma))$.
\end{itemize}
 \end{remark}

 We are now producing explicit examples of groups satisfying the conditions of Remark \ref{RmkNotWA}

 Let $\Lambda$ and $\Sigma$ be a countable groups and define $\Gamma=\Sigma\wr\Lambda:=\Sigma^{(\Lambda)}\rtimes\Lambda$, where $\Sigma^{(\Lambda)}$ denotes $\bigoplus_\Lambda\Sigma$ and the $\Lambda$ action is by left translation of coordinates.

 \begin{lem}\label{WreathContinuity}
     The map ${\rm Sub}(\Sigma)\rightarrow{\rm Sub}(\Gamma)$, $G\mapsto G\wr\Lambda$ is continuous.
 \end{lem}

 \begin{proof}
Note that if $H_1$ is a subgroup of $H_2$ then the canonical map ${\rm Sub}(H_1)\rightarrow {\rm Sub}(H_2)$ is clearly continuous. Hence, it suffices to show that the map ${\rm Sub}(\Sigma)\rightarrow {\rm Sub}(\Sigma^\Lambda)$, $G\mapsto G^{(\Lambda)}$ is continuous. Fix $g=(g_\lambda)_{\lambda\in\Lambda}\in G^{(\Lambda)}$ and let $F\subset\Lambda$ be a finite subset such that $g_\lambda=1$ for all $\lambda\notin F$. Note that $1_{G^{(\Lambda)}}(g)=\prod_{\lambda\in F}1_G(g_\lambda)$ which implies that $G\mapsto 1_{G^{(\Lambda)}}(g)$ is continuous.\end{proof}

 \begin{prop}\label{PropExampleWeaklyAmenable}
 Suppose that the following holds.
 \begin{itemize}
     \item $\Lambda$ is non-amenable but weakly amenable.
     \item The trivial subgroup is not isolated in ${\rm Sub}(\Sigma)$.
 \end{itemize}
Then, the $\Lambda_{cb}$ map is not continuous on ${\rm Sub}(\Gamma)$ and the set of weakly amenable subgroups of $\Gamma$ is not open in ${\rm Sub}(\Gamma)$.
 \end{prop}
 \begin{proof}
Since $\{1\}\in{\rm Sub}(\Sigma)$ is not isolated, there exists a sequence of non-trivial subgroups $\Sigma_n<\Sigma$ such that $\Sigma_n\rightarrow\{1\}$. By Lemma \ref{WreathContinuity}, we have $\Sigma_n\wr\Lambda\rightarrow \Lambda$. However, since each $\Sigma_n$ is non-trivial, \cite[Corollary 4]{Oz12} implies that $$\Lambda_{cb}(\Sigma_n\wr\Lambda)=+\infty\text{ for all }n.$$ Since $\Lambda$ is assumed to be weakly amenable, it ends the proof.
 \end{proof}

 \begin{example}
     An example of the situation above is $\Lambda=\mathbb{F}_n$ the free group on $n\geq 2$ generators and $\Sigma=\Z$. The trivial subgroup $\{0\}<\Z$ is not isolated since it is easily seen to be the limit of the sequence $(p^n\Z)_n$, for any prime number $p$.
 \end{example}

 Another family of examples is described in the next proposition.

 \begin{prop}
     Let $\Gamma_k$ be a sequence of groups such that $\inf_{k\in\N}\Lambda_{cb}(\Gamma_k)>1$ and define $\Gamma:=\oplus_k\Gamma_k$. Then, the $\Lambda_{cb}$ map is not continuous on ${\rm Sub}(\Gamma)$. If we assume moreover that $\Gamma_k$ is weakly amenable for all $k$ then, the set of weakly amenable subgroups of $\Gamma$ is not closed in ${\rm Sub}(\Gamma)$.
 \end{prop}

 \begin{proof}
We first observe that, viewing each $\Gamma_k\in{\rm Sub}(\Gamma)$ in the obvious way, we have $\Gamma_k\rightarrow \{1\}$. Indeed, let $g=(g_k)_k\in\Gamma\setminus\{1\}$ and let $i$ be such that $g_i\neq 1$. Then, for all $k>i$ one has $1_{\Gamma_k}(g)=0$. Since $\Lambda_{cb}(\{1\})=1$, this shows that $\Lambda_{cb}$ is not continuous. Next, consider the sequence of subgroups defined by $$\Gamma_{\leq n}:=\oplus_{k=1}^n\Gamma_k\simeq \Gamma_1\times\dots\times\Gamma_n.$$
Then, by \cite{CH89}, one has, for all $n\in\N$, $\Lambda_{cb}(\Gamma)\geq\Lambda_{cb}(\Gamma_{\leq n})=\prod_{k=1}^n\Lambda_{cb}(\Gamma_k)\geq C^n$, where $C:=\inf_{k\in\N}\Lambda_{cb}(\Gamma_k)>1$. It follows that $\Lambda_{cb}(\Gamma)=+\infty$. To conclude the proof, it suffices to show that $\Gamma_{\leq n}\rightarrow\Gamma$. Let $g=(g_k)_k\in\Gamma$ and $F\subset\N$ be a finite subset such that $g_k=1$ for all $k\notin F$. Then, for all $n>\max F$, one has $1_{\Gamma_{\geq n}}(g)=1$. It shows that $\Gamma_{\leq n}\rightarrow\Gamma$ and concludes the proof.\end{proof}
 
\begin{example}
To get an explicit example of the situation above one could, by \cite{CH89}, take the groups $\Gamma_k$ to be any lattices in ${\rm Sp}(1,n)$ $(n\geq 2)$ or in ${\rm F}_{4(-20)}$  .
\end{example}

\subsection{Applications to Kechris spaces of nonsingular equivalence relations}

We now briefly mention applications of our results to non-singular equivalence relations.
For a more detailed account, the reader should consult \cite{lemaitreDenseorbits2026}. 

Given a (countable) non-singular equivalence relation $\mathcal R$ on a standard 
probability space $(X,\mu)$, define its \textbf{Kechris space} as the space 
$\Sub(\mathcal R)$ of (measurable) subequivalence relations up to a null set. 
Recall that $\mathcal R$ can be endowed with the (left) counting measure 
$M$ defined by $M(A)=\int_X \abs{A_x}d\mu(x)$, which is $\sigma$-finite. 
The measure algebra of $(\mathcal R,M)$ (space of Borel subsets of $\mathcal R$ up to a null set)
is thus a Polish space, a compatible complete metric being given by 
$d_m(A,B)=m(A\bigtriangleup B)$ where $m$ is a fixed probability measure equivalent to $M$.
We can then view $\Sub(\mathcal R)$ as a subset of the measure algebra of $(\mathcal R,M)$, 
and it is actually closed therein, hence Polish for the induced topology, which 
we call the \textbf{Kechris topology}. 

Given a non-singular equivalence relation $\mathcal R$, 
one defines its von Neumann algebra $\LL\mathcal R$, and 
the map $\Sub(\mathcal R)\to \mathcal S(\LL\mathcal R)$ is a homeomorphism
onto its image, which is exactly the set of subalgebras containing $\LL^\infty(X,\mu)$, 
and only consists of subalgebras which are expected upon
with respect to the canonical state $\varphi$ associated to the measure $M$,
namely $\varphi(x)=\la x 1_{\Delta X}, 1_{\Delta_X}\ra_{\LL^2(\mathcal R,M)}$
(see \cite[Section~7]{zhouNoncommutativeTopologicalBoundaries2025}).

    In \cite{zhouNoncommutativeTopologicalBoundaries2025} it is mentioned that Kechris' result 
that the
space of amenable subequivalence relations of a given pmp equivalence relation is closed
also follows from the fact that amenable subalgebras form a closed subset in a finite von Neumann algebra,
and that an equivalence relation is amenable iff its von Neumann algebra is amenable. 
Let us remark here that the same argument works in the non pmp case 
as a consequence of our Proposition~\ref{prop: amenable is closed v1}
and of the fact that von Neumann algebras of subequivalence relations are contained in $\mathcal E_{\varphi}$:

\begin{prop}
    The space of amenable subequivalence relations of a given non-singular equivalence relation $\mathcal R$
    is closed in the Kechris space $\Sub(\mathcal R)$. \qed
\end{prop}

We now make similar remarks concerning the Haagerup property (see \cite{kidaSplittingOrbitEquivalence2017} for the definition). 
Indeed, by \cite{anantharaman-delarocheHaagerupPropertyDiscrete2013}, a non-singular equivalence relation $\mathcal S$ has 
the Haagerup property if and only if $\LL\mathcal S$ has the Haagerup property relative to $\LL^\infty(X,\mu)$. 
Combining this with \cite[Theorem~7.2]{zhouNoncommutativeTopologicalBoundaries2025}, we arrive at the following result.

\begin{prop}
    Given a non-singular equivalence relation $\mathcal R$,
    the space of subequivalence relations with the Haagerup property 
    is closed in  $\Sub(\mathcal R)$. \qed
\end{prop}

\begin{remark}
The fact that the space of amenable subequivalence relations is always closed 
was originally proven by Kechris in the pmp case by combining the following two facts
(also true in the non-singular case, see \cite{lemaitreDenseorbits2026} for details): 
\begin{itemize}
    \item Every converging sequence of equivalence relations actually converges 
    to the liminf of one of its subsequences
    \item The class of amenable equivalence relations is closed under directed 
    countable unions and taking subequivalence relations. 
\end{itemize}
A similar approach can be used to recover our result on the Haagerup property.
Similar applications may also be expected for weak amenability, however the corresponding theory for equivalence relations, and more generally for groupoids, is not yet sufficiently developed \cite{pachecoWeaklyAmenableGroupoids2025}, so we have not pursued them here.
\end{remark}

\section{On subalgebras without conditional expectation}\label{sec:no_cond_exp}

Let $\mathcal{E}_\varphi^{c}$ denote set of von-Neumann subalgebras of $M$ which lack $\varphi$-invariant conditional expectations i.e the complement of $\mathcal{E}_{\varphi}$ in $\mathcal{S}(M)$. We shall use the following version of the characterization of finite von Neumann algebras due to Takesaki \cite[Section 6]{takesakiConditionalExpectationsNeumann1972}.

\begin{lem}\label{LemFinite}
$M$ is finite if and only if $\underset{\varphi\text{ nfs on }M}{\bigcap}\mathcal{E}_\varphi^c=\emptyset$.
\end{lem}

\begin{proof}
Since any von Neumann subalgebra of a finite von Neumann subalgebra $M$ is in $\mathcal{E}_\tau$, where $\tau$ is a faithful normal trace on $M$, one implication is clear. Assume now that $M$ is not finite. By \cite[Section 6]{takesakiConditionalExpectationsNeumann1972} there exists a MASA $A\subset M$ such that there is no normal conditional expectation $M\rightarrow A$. In particular, $A\notin\mathcal{E}_\varphi$, for any nfs $\varphi\in M_*$.
\end{proof}

In the sequel, we denote by $M_\varphi:=\{x\in M\,:\,\sigma^\varphi_t(x)=x\text{ for all }t\in\R\}$ the centraliser of a nfs $\varphi$ on $M$.

\begin{thm}\label{Nocondexpecdense}
Let $M$ be a purely infinite von Neumann algebra. Then, for any nfs $\varphi\in M_*$ such that $M_\varphi$ is infinite dimensional the set $\mathcal{E}_\varphi^c$ is open and dense in $\mathcal{S}(M)$.   
\end{thm}

\begin{proof}
We already proved that $\mathcal{E}_\varphi^c$ is open (see Theorem~\ref{thm: marechal same as one lip}). Since $\mathcal{S}(M)$ is a Polish space, it is enough to show that if $B\in \mathcal{S}(M)$, then there exists $B_n\in \mathcal{E}_\varphi^c$ such that $B_n\rightarrow B$.

Let $q_n\in M_\varphi$, $n\in \N$, be an increasing sequence of projections such that $q_n\rightarrow 1$ strongly. Let $p_n=1-q_n$, $n\in\N$. Let $\varphi_{p_n}(.)=\frac{\varphi(p_n . p_n)}{\varphi(p_n)}$ and 
$\varphi_{q_n}(.)=\frac{\varphi(q_n . q_n)}{\varphi(q_n)}$ denote the natural reduced faithful normal states on the von Neumann algebras $p_nMp_n$ and $q_nMq_n$ respectively. 

Since $M$ is purely infinite, $p_nMp_n$ is not finite and, by Lemma \ref{LemFinite}, there exists $A_n\subset p_nMp_n$ such that $A_n\in \mathcal{E}_{\varphi_{p_n}}^c$. Consider now $B_n=A_n\oplus q_nBq_n$, which is a unital subalgebra of $M$. It can be checked that $B_n\in \mathcal{E}_\varphi^c$. To see this, first we note that since $p_n\in M_{\varphi}$, the modular group of $\varphi_{p_n}$ is given by $\sigma_t^{\varphi_{p_n}}(.)=p_n\sigma_t^{\varphi}(.)p_n$, for every $n\in\N$ and $t\in\R$. So, suppose now that $B_n\in \mathcal{E}_{\varphi}$, then $B_n$ is invariant under the action of the automorphism group $\sigma_t^{\varphi}$. In particular then, $\sigma_t^\varphi(A_n)\subseteq B_n$ but since $p_nq_n=0$, we have that $A_n$ is invariant under the action of the automorphism group $\sigma_t^{\varphi_{p_n}}$ on $p_nMp_n$, implying that $A_n\in \mathcal{E}_{\varphi_{p_n}}$, which is a contradiction. Hence, we have $B_n\in \mathcal{E}_\varphi^c$ for all $n$.

Finally, let us show that $B_n\rightarrow B$. Let $\omega\in M_*$, take $x\in B_1$ such that $\Vert\omega\vert_B\Vert=\vert\omega(x)\vert$ and define $x_n := p_n\oplus q_nxq_n\in B_n$. Then, $x_n\rightarrow x$ strong$^\ast$ so $$\Vert\omega\vert_B\Vert=\vert\omega(x)\vert=\lim\vert\omega(x_n)\vert\leq\liminf\Vert\omega\vert_{B_n}\Vert.$$

Next, let $x_n\in (B_n)_1$ be such that $\Vert\omega\vert_{B_n}\Vert=\vert\omega(x_n)\vert$. Let $\phi\,:\,\N\rightarrow\N$ be an increasing map such that $\vert\omega(x_{\phi(n)})\vert\rightarrow\limsup\Vert\omega\vert_{B_{n}}\Vert$. Since $M_1$ is weakly compact, we may and will assume that $x_{\phi(n)}\rightarrow x\in M_1$ weakly. Fix $y\in B'$ and define $y_n = p_n\oplus yq_n\in B_n^\prime$, then $y_n\rightarrow y$ in the strong$^\ast$ topology. Since the multiplication $\mathbf{B}(H)_1\times \mathbf{B}(H)_1\rightarrow \mathbf{B}(H)_1$ is both strong$\times$weak to weak and weak$\times$strong to weak continuous, it follows that $x_ny_n\rightarrow xy$ and $y_nx_n\rightarrow yx$ weakly. Hence, $xy=yx$ and since this holds for all $y\in B'$ we deduce that $x\in B''=B$. It implies that:
$$
\Vert\omega\vert_B\Vert\geq\vert\omega(x)\vert=\lim\vert\omega(x_{\phi(n)})\vert=\limsup\Vert\omega\vert_{B_n}\Vert.\qedhere
$$
\end{proof}

\begin{remark}\label{remarkoncentralizerhypothesis}
Note that Theorem \ref{Nocondexpecdense} is valid for all type $\rm{III}_\lambda$ factors for $\lambda\in [0,1)$ and for all faithful normal states \cite{Co73}. It is also valid for type $\rm{III}_1$ factors with faithful normal states having infinite dimensional centralisers. However, it is known that the generic state on a type ${\rm III}_1$ factor is faithful and has trivial centraliser \cite[Theorem A]{MV24}.\end{remark}

\begin{thm}\label{thm: no condex is dense in finite}
Let $M$ be a von Neumann algebra with separable predual. Put
$$
        \mathrm{Exp}(M)
        =
        \{N\in \mathcal{S}(M):\text{ there exists a normal conditional expectation }
        M\to N\}.
$$
The following are equivalent:
\begin{enumerate}
    \item $\mathrm{Exp}(M)^c$ is dense in $\mathcal{S}(M)$,
    \item $M$ is not finite.
    \end{enumerate}
\end{thm}

\begin{proof}$(1)\Rightarrow(2)$ is clear. Indeed, if $M$ is finite then $\mathrm{Exp}(M)^c=\emptyset$ is not dense. Let us show $(2)\Rightarrow(1)$. If $M$ is not finite then there is a nonzero central projection $z\in \mathcal{Z}(M)$ such that
$zM$ is properly infinite. Since $M$ has separable predual, $zM$ is
countably decomposable.
Since $zM$ is properly infinite and countably decomposable, there exists a von Neumann algebra $P$ such that $zM\simeq \textbf{B}({\rm L}^2(\mathbb R_+))\otimes P$. Under this identification, define $p_n:=1_{(0,1/n)}\otimes 1_P$ and $q_n:=z-p_n=1_{(1/n,\infty)}\otimes 1_P$. Let $S_t$ be the right-shift isometry on ${\rm L}^2(\mathbb R_+)$,
$$
        (S_t\xi)(s)=
        \begin{cases}
        0, & 0<s<t,\\
        \xi(s-t), & s\ge t.
        \end{cases}
$$
Then $S_t^*S_t=1$ and $S_tS_t^*=1_{(t,\infty)}$. Moreover, $S_t\to1$, $S_t^*\to1$ strongly as $t\to0$. Put $w_n:=S_{1/n}\otimes1_P\in zM$ then $w_n^*w_n=z$, $w_nw_n^*=z-p_n$ and $w_n\to z$, $w_n^*\to z$ strongly. Now define $v_n:=(1-z)+w_n\in M$ then $v_n^*v_n=1$, $v_nv_n^*=1-p_n$ and $v_n\to1$, $v_n^*\to1$ strongly.

\medskip

\noindent Since each corner $p_nMp_n$ is properly infinite, hence non-finite,
we may choose, by \cite[Section 6]{takesakiConditionalExpectationsNeumann1972}, $A_n\subset p_nMp_n$ such that there is no normal conditional expectation $p_nMp_n\to A_n$. Now fix $N\in \mathcal{S}(M)$. Define $q_n:=1-p_n$ and,
$$
        N_n:=v_nNv_n^*\oplus A_n
        \subset q_nMq_n\oplus p_nMp_n=M.
$$
Thus $N_n$ is a unital von Neumann subalgebra of $M$. We claim first that $N_n\notin \mathrm{Exp}(M)$ for all $n$. Indeed, suppose that there were a faithful normal conditional expectation $E:M\to N_n$. Since $p_n\in \mathcal{Z}(N_n)$, the compression
$$E_n:p_nMp_n\to p_nN_np_n=A_n,\quad E_n(x):=p_nE(x)p_n$$
is a normal conditional expectation. This contradicts the choice of $A_n$.
Therefore $N_n\notin \mathrm{Exp}(M)$. It remains to prove that $N_n\to N$. Fix $\omega\in M_*$. Since $N_n$ is the direct sum of its two central corners
$q_nN_nq_n=v_nNv_n^*$ and $p_nN_np_n=A_n$, we have
$$
        \|\omega|_{N_n}\|
        =
        \|\omega|_{v_nNv_n^*}\|
        +
        \|\omega|_{A_n}\|.
$$
For the first term, $\|\omega|_{v_nNv_n^*}\| = \|(\omega\circ\operatorname{Ad}v_n)|_N\|$ where, $(\omega\circ\operatorname{Ad}v_n)(x)=\omega(v_nxv_n^*)$. Since $v_n\to1$ and $v_n^*\to1$ strongly, and since $\omega$ is normal, $\omega\circ\operatorname{Ad}v_n\to \omega$ in norm in $M_*$. Hence $\|\omega|_{v_nNv_n^*}\|=\|(\omega\circ\operatorname{Ad}v_n)|_N\|\to \|\omega|_N\|$. For the second term, since $A_n\subset p_nMp_n$, we have $\|\omega|_{A_n}\|\leq\|\omega|_{p_nMp_n}\|$. Because $p_n\downarrow0$ strongly and $\omega$ is normal, $\|\omega|_{p_nMp_n}\|\to0$. Indeed, if $\omega\geq0$, then $\|\omega|_{p_nMp_n}\|=\omega(p_n)\to0$, and the general case follows by decomposing $\omega$ as a linear combination
of positive normal functionals. Therefore $\|\omega|_{N_n}\|=\|\omega|_{v_nNv_n^*}\|+\|\omega|_{A_n}\|\to\|\omega|_N\|$. Thus $N_n\to N$ in $\mathcal{S}(M)$.\end{proof}

\begin{remark}As the set $\mathrm{Exp}(M)$ is clearly analytic (in the sense of descriptive set theory, see 
\cite[Chapter~3]{kechrisClassicalDescriptiveSet1995}), it is natural to ask whether it is a Borel set when $M$ is not finite, 
and whether it is generic.
\end{remark}

\bibliographystyle{alphaurl}
\bibliography{biblio}

\end{document}